\documentclass[12pt,reqno]{amsart}
\usepackage{amsmath}
\usepackage{amssymb}
\usepackage{amstext}
\usepackage{a4wide}

\usepackage{graphicx}
\usepackage{tikz}
\usepackage{subfigure}

\allowdisplaybreaks \numberwithin{equation}{section}
\usepackage{color}
\usepackage{cases}

\numberwithin{equation}{section}

\newtheorem{theorem}{Theorem}[section]
\newtheorem{proposition}[theorem]{Proposition}
\newtheorem{corollary}[theorem]{Corollary}
\newtheorem{lemma}[theorem]{Lemma}

\theoremstyle{definition}

\newtheorem{definition}[theorem]{Definition}

\theoremstyle{remark}

\begin{document}

\title
[Steady vortex-patch flows with circulation ]{Steady vortex-patch  flows with circulation past an obstacle in an infinity long strip}

\author{Weilin Yu}
\address{School of Mathematics and Statistics, Jiangxi Normal University, Nanchang, Jiangxi 330022, P. R. China
\\
Beijing Institute of Mathematical Sciences and Applications,
    Beijing 101408,
    P. R. China}
\email{weilinyu@amss.ac.cn}


\begin{abstract}
We consider the  steady Euler flows  past an  obstacle in an infinity long strip with  horizontal constant velocity at infinity, 
prescribed circulation around the obstacle and sharply concentrated patch-type vorticity.  The construction of these flows are based on the structure of a new Green's function, the existence of stable minimum points of the   Kirchhoff-Routh functions and the existence of maximizers of the kinetic energy for the vorticity. We mainly focus on the effect from the obstacle, the velocity at infinity and the circulation around the obstacle on the existence of  minimum points of the  Kirchhoff-Routh function, and  hence on the existence of  vortex patch flows.  
\end{abstract}

\maketitle


\section{Introduction and main results}

The problem of vortex flow past  obstacles is an important issue in fluid dynamics, involving the interactions between fluids and solid obstacles. When a fluid flows past an obstacle, various types of vortex phenomena may exist, such as the F\"oppl vortex pair \cite{F,VMS},   Kármán vortex street \cite{HRHB,vK}, horseshoe vortices \cite{ONS}, wake vortices \cite{PPS},  and turbulent eddies \cite{S}. 
The  theory  of  vortex phenomena in flow around obstacles  has likewise attracted great attentions in mathematics community. We refer readers to see \cite{EFHM,EM,Tur83} for the existence of steady vortices, 
\cite{M} for the evolution of vortex set and \cite{HI,ILN1,ILN2,L,LM} for the asymptotic behaviors as the size of obstacles tends to zero.


In this paper we consider a steady, incompressible Euler flow  past an obstacle in a two-dimensional infinitely long strip:
\[
\Omega=S\setminus \bar{O}_0,
\]
where $S=\mathbb{R}\times (0,\pi)$, and  $O_0$ is a bounded, simply-connected open subset of $\mathbb{R}^2$ with $C^3$ boundary such that $\bar{O}_0\subset S$. 
This flow is governed by the  following incompressible Euler equation
\begin{equation}\label{eq1-1}
\begin{cases}
(\mathbf{v}\cdot \nabla)\mathbf{v}+\nabla P=0\quad &\text{in}\; \Omega,\\
\text{div}\; \mathbf{v}=0\quad &\text{in}\; \Omega,\\
\mathbf{v}\cdot \mathbf{n}=0\quad &\text{on}\; \partial \Omega,
\end{cases}
\end{equation}
where $\mathbf{v}=(v_1,v_2)$ and $P$ are the velocity and pressure of this flow, respectively, $\mathbf{n}=(n_1,n_2)$ is the unit outward normal of $\partial \Omega$.  We are interested in  a special type of Euler flows, whose velocity is uniform and horizontal at infinity and circulation around the obstacle is prescribed, namely, 
\begin{equation}\label{eq1-2}
\mathbf{v}\to \mathbf{v}_\infty:=(-b,0)\;\;\text{as}\;\; |x|\to \infty,\quad \oint_{\partial O_0} \mathbf{v}\cdot \tau\; dS_x= \Gamma, 
\end{equation}
where $b, \Gamma \in \mathbb{R}\setminus \{0\}$ are two constants, $\tau=(n_2,-n_1)$ is the unit tangent of $\partial O_0$.

In this flow, the stream function $\psi$ is defined by $\mathbf{v}=(\partial_2 \psi, -\partial_1 \psi)$, and the scalar vorticity is defined by $\omega=\partial_1v_2-\partial_2v_1$, where $\partial_i=\frac{\partial }{\partial x_i}$ for $i=1,2$. By applying \textit{curl} on the first equation of \eqref{eq1-1},  we get the following equation 
\begin{equation}\label{eq1-3}
\nabla \omega \cdot \nabla^\perp \psi=0\;\;\;\text{in}\;\; \Omega.
\end{equation}
where $(a_1,a_2)^\perp=(a_2,-a_1)$.  Moreover,  under the sliding boundary condition $\mathbf{v}\cdot \mathbf{n}=0$  on $\partial \Omega$ and the condition \eqref{eq1-2}, 
$(\omega,\psi)$ should further satisfy the following elliptic equation 
\begin{equation}\label{eq1-4}
\begin{cases}
-\Delta \psi=\omega\quad \text{in}\; \Omega,\\
\psi=constant \quad \text{on each connected component of $\partial \Omega$}, 
\\
\nabla \psi \to (0,-b)\;\; \text{as}\;\; |x|\to \infty, 
\\
\int_{\partial O_0} \frac{\partial \psi}{\partial \mathbf{n}}\; dS_x= \Gamma. 
\end{cases}
\end{equation}
 Let us introduce a streamfunction $-\eta$ of an irrotational flow with velocity $(-b,0)$ and circulation $\Gamma$ around $\partial O_0$, and is defined by 
\begin{equation}\label{eq1-5}
\begin{cases}
\Delta \eta=0\quad \text{in}\; \Omega,\\
\eta=-\lambda \;\;  \text{on}\; \partial O_0,\quad \eta=0 \;\; \text{on}\;\; \{x_2=0\}, \quad  \eta=b\pi \;\; \text{on}\;\; \{x_2=\pi\}, 
\\
\nabla \eta \to (0,b)\;\; \text{as}\;\; |x|\to \infty, 
\\
\int_{\partial O_0} \frac{\partial \eta}{\partial \mathbf{n}}\; dS_x=-\Gamma, 
\end{cases}
\end{equation}
where $\lambda$ is a flux constant as a Lagrange multiplier.  Suppose that $\omega\in L^2(\Omega)$. Then the following equation has a unique solution $u=\mathcal{G}\omega\in H^1(\Omega)$:
\begin{equation}\label{eq1-6}
\begin{cases}
-\Delta u=\omega \quad \text{in}\; \Omega,\\
u=\text{constant} \;\;  \text{on}\; \partial O_0,\;\;\; u=0 \;\; \text{on}\; \partial S, 
\\
u \to 0\;\; \text{as}\;\; |x|\to \infty, 
\\
\int_{\partial O_0} \frac{\partial u}{\partial \mathbf{n}}\; dS_x=0,
\end{cases}
\end{equation}
where $\mathcal{G}$ is the Green's operator (see \eqref{eq1-13} for the definition). If  $\omega$ additionally has compact support, then $\nabla \mathcal{G}\omega (x)\to 0$ as $|x|\to +\infty$. Thus, 
\begin{equation}\label{eq1-7}
\psi=\mathcal{G}\omega-\eta-\mu
\end{equation}
is a solution of \eqref{eq1-4}, where $\mu$ is a constant.

Based on above discussions,  by inserting  \eqref{eq1-7} into \eqref{eq1-4} we can translate the  problem  \eqref{eq1-1}-\eqref{eq1-2} into  
the following vorticity equation 
\begin{equation}\label{eq1-8}
\nabla \omega \cdot \nabla^\perp (\mathcal{G}\omega-\eta)=0\;\;\;\text{in}\;\;\Omega,
\end{equation}
where $\omega$ should satisfy the  condition
\begin{equation}\label{eq1-9}
\omega\in L^2(\Omega),\quad \overline{\text{supp}(\omega)}\;\; \text{is a compact set}. 
\end{equation}
Since $\mathcal{G}$ is a linear continuous map from $L^2(\Omega)$ to $H^1(\Omega)$, we can give the definition of weak solutions to the vorticity equation \eqref{eq1-8}. 

\begin{definition}\label{de2-1} We call $\omega\in L^2(\Omega)$  a weak solution of \eqref{eq1-8} if for all $\phi\in C_0^\infty(\Omega)$, 
\begin{equation}\label{eq1-10}
\int_\Omega  \omega \nabla^\perp (\mathcal{G}\omega-\eta)\cdot \nabla \phi \;dx=0.
\end{equation}
\end{definition}

A typical weak solution whose vorticity is the characteristic function of  finite isolated bounded regions with jump discontinuity, is called \textbf{vortex patch solution}. Such patterns are a special class of  solutions satisfying \eqref{eq1-9}, and has the form 
\begin{equation}\label{eq1-11}
\omega(x)=\frac{1}{\varepsilon^2}\sum_{i=1}^k \mathbf{1}_{A_i},
\end{equation}
where $A_i$, $i=1,\cdots,k$ are $k$ disjoint connected bounded subdomains of $\Omega$, $\mathbf{1}_A$ is the characteristic function of $A$, namely, $\mathbf{1}_A=1$ in $A$, otherwise, $\mathbf{1}_A=0$,  and $1/\varepsilon^2$ is the vorticity strength parameter. 

There are a lot of literatures for vortex patch solutions of Euler equations in   planar domains.  For  bounded domains,  Turkington \cite{Tur83} has proved  the existence of vortex patch solutions with single vortex $(k=1)$ via the vorticity method, and  showed that as $\varepsilon\to 0$, the vorticity concentrated at a local minimum point of the Robin function.  Cao-Peng-Yan \cite{CPY2} further construct vortex patch solutions with multiple vortices $(k\geq 2)$ via  the so-called  Lyapunov-Schmidt reduction method, where as $\varepsilon\to0$, the vorticity concentrates at some isolated  critical point with non-zero Brouwer degree of the Kirchhoff-Routh function. The local uniqueness of these  solutions  are studied in \cite{CGPY}. 
For  unbounded domains,  the  only known  explicit simply-connected stationary vortex patch solution is called the Rankine vortex, which is defined in the whole plane. In \cite{Tur83}, Turkington  considered the Euler flow past a bounded axially symmetric obstacle in the whole plane, and proved that there exist  solutions having a pair of vortex patches with opposite signs. Furthermore, Elcrt and Miller \cite{EM} studied the case of  asymmetric obstacle,  and obtained the existence of single vortex patch solutions if the   circulation  around the obstacle is large enough.  As for rotating vortex patch problems, please to see  \cite{CCG,HHHM,HHMV,HM,K} and references therein.

In this paper, we are interested in the existence of vortex patch solutions of \eqref{eq1-8} in an infinite long strip with an obstacle. To  this purpose, we first study Green's function of the following 
linear problem 
\begin{equation}\label{eq2-1}
\begin{cases}
-\Delta u=f\quad \text{in}\; \Omega,\\
u=\text{constant} \;\;  \text{on}\; \partial O_0,
\quad
u=0 \;\; \text{on}\; \partial S, 
\\
u \to 0\;\; \text{as}\;\; |x|\to \infty, 
\\
\int_{\partial O_0} \frac{\partial u}{\partial \mathbf{n}}\; dS_x=0.
\end{cases}
\end{equation}
By Lemma~\ref{lem2-1y},  for any $f\in L^2(\Omega)$,   there exists a unique weak solution $u_f\in H^1(\Omega)$ of \eqref{eq2-1}.  So, one can define a Green's operator 
\begin{equation}\label{eq1-13}
\mathcal{G}: L^2(\Omega) \to H^1(\Omega),\quad f\mapsto u_f.
\end{equation}

The following   result gives a structure of the Green's function of \eqref{eq2-1}. 
\begin{theorem}\label{thm2-2} 
 For every $f\in L^2(\Omega)$ and every $x\in \Omega$, 
\begin{equation}\label{eq2-2}
\mathcal{G} f(x)=\int_{\Omega} \big(G_0(y,x)+\lambda_0 \rho(y)\rho(x)\big) f(y)\; dy,
\end{equation}
where  $G_0(y,x)$ is the Green function of $-\Delta$ in $\Omega$ with zero Dirichlet boundary condition, 
 $\rho\in C^2(\bar{\Omega})$ is the unique solution of 
\begin{equation}\label{eq2-3}
\begin{cases}
\Delta \rho=0\quad \text{in}\; \Omega,\\
\rho=1 \;\;  \text{on}\; \; \partial O_0,
\;\;\;
\rho=0 \;\; \text{on}\;\;  \partial S, 
\\
\rho, |\nabla \rho|\to 0\;\; \text{as}\;\; |x|\to \infty,
\end{cases}
\end{equation}
and 
$
\lambda_0=\Big(\int_{\partial O_0}\; \frac{\partial \rho}{\partial \mathbf{n}}\; dS_x\Big)^{-1}>0.
$
\end{theorem}

As a consequence  of Theorem~\ref{thm2-2}, the Green's function $G(y,x)$ of  \eqref{eq2-1}  has the following expansion 
\begin{equation}\label{eq2-4}
G(y,x)=G_0(y,x)+\lambda_0 \rho(x) \rho(y).
\end{equation}
Let $H(y,x)$ be the regular part of $G(y,x)$ so that $G(y,x)=\frac{1}{2\pi}\ln \frac{1}{|y-x|}-H(y,x)$. The 
 following function, so-called  Kirchhoff-Routh function, plays an important role in the study of concentrated  single-vortex solution of  Euler equations
\begin{equation}\label{eq1-17}
\mathcal{H}(x)= H(x,x)+2 \eta(x).
\end{equation}

\begin{theorem}\label{th3-1}Let  $b$  and $\Gamma$  be two positive constants. There holds:
  \begin{itemize}
  
   \item[$(i)$] If $\Gamma/b \geq \pi \int_{\mathbb{R}\times \{x_2=\pi\}} | \frac{\partial \rho}{\partial \mathbf{n}} |\; dS_x$, then the global minimum set $\mathcal{Q}$ is a non-empty bounded set and 
   \[
   \text{dist} (\mathcal{Q},\partial \Omega)>0,
   \]
   where $\mathcal{Q}$ denotes all the global minimum points of $\mathcal{H}$ on $\bar{\Omega}$.

  \item[$(ii)$] There exist two positive  constants $\theta_2>\theta_1>0$ and  a large constant $b_0>0$ such that if $b>b_0$ and  $\Gamma/b \geq \pi \int_{\mathbb{R}\times \{x_2=\pi\}} | \frac{\partial \rho}{\partial \mathbf{n}} |\; dS_x$, then the local minimum set $\mathcal{Q}_{1}$ is  non-empty  and 
 \[
 \text{dist} (\mathcal{Q}_{1}, \partial  D_{b,\Gamma} )>0,
 \]
 where $\mathcal{Q}_{1}$  denotes all the minimum points of $\mathcal{H}$ on $\overline{D}_{b,\Gamma}$ and 
\[
D_{b,\Gamma}:=\Big\{ x\in \Omega:\; \frac{\theta_1}{b+\lambda_{b,\Gamma}}<dist (x,\partial O_0)< \frac{\theta_2}{b+\lambda_{b,\Gamma}}
\Big\}.
\]
Here, $\lambda_{b,\Gamma}$ is a positive constant depending on $b$ and $\Gamma$ (see \eqref{eq2-24}).

  \item[$(iii)$] There exist two small constant $\sigma>0$ and $\delta>0$,  and  a large constant $b_0>0$ such that if $b>b_0$ and  
  \[
  \pi \int_{\mathbb{R}\times \{x_2=\pi\}} | \frac{\partial \rho}{\partial \mathbf{n}} |\; dS_x\leq \Gamma/b\leq  \pi \int_{\mathbb{R}\times \{x_2=\pi\}} | \frac{\partial \rho}{\partial \mathbf{n}} |\; dS_x+\sigma \int_{\partial O_0} |\frac{\partial \rho}{\partial \mathbf{n}}| \; dS_x,
  \]
 then the local minimum sets $\mathcal{Q}_{1}$ and $\mathcal{Q}_{2}$ are two non-empty bounded sets and 
 \[
 \text{dist} (\mathcal{Q}_{1}, \partial  D_{b,\Gamma})>0,\;\; \text{dist} (\mathcal{Q}_{2}, \partial  \Omega_\delta)>0,
 \]
 where $\mathcal{Q}_{1}$  denotes all the minimum points of $\mathcal{H}$ on $\bar{D}_{b,\Gamma}$  and  $\mathcal{Q}_{2}$  denotes all the minimum points of $\mathcal{H}$ on $\overline{\Omega}_\delta$, and 
 \[
\Omega_\delta:=\{x\in \Omega:\; \text{dist}(x,\partial O_0)>\delta\}.
 \]

\end{itemize}
\end{theorem}

A key step to prove Theorem~\ref{th3-1} is to show the asymptotic behaviors of $\mathcal{H}$  at infinity and near the obstacle $O_0$.  The first part  is achieved  through the following Green's representation formula 
\begin{equation}\label{eq1-18e}
f(x)=\int_{ \partial O_0} G_S(y,x)\frac{\partial f(y)}{\partial \mathbf{n}}-f(y)\frac{\partial G_S(y,x)}{\partial \mathbf{n}} dS_y,\quad x\in \Omega,
\end{equation}
where $f$ is an arbitrary bounded harmonic function in $\Omega$ such that $f=0$ on $\partial S$, and $G_S(y,x)$ is the Green function of $-\Delta$ in $S$ with zero Dirichlet boundary condition, written as (see \cite{CGO})
\begin{equation}\label{eq1-18}
\begin{split}
G_S(y,x)=-\frac{1}{4\pi} \ln \big( \cosh(y_1&-x_1)-\cos(y_2-x_2)\big)
\\
&+\frac{1}{4\pi} \ln \big( \cosh(y_1-x_1)-\cos(y_2+x_2)\big).
\end{split}
\end{equation}

As applications of  Theorem~\ref{thm2-2} and Theorem~\ref{th3-1}, our first main result of this paper can be stated as follows. 
\begin{theorem}\label{thm1-4}Let  $b$  and $\Gamma$  be two positive constants, such that $\Gamma/b \geq \pi \int_{\mathbb{R}\times \{x_2=\pi\}} | \frac{\partial \rho}{\partial \mathbf{n}} |\; dS_x$. Then there exists $\varepsilon_0>0$ such that for each $0<\varepsilon<\varepsilon_0$,  there exists a solution $\omega_\varepsilon$ of \eqref{eq1-8} with the form 
\[
\omega_\varepsilon=\frac{1}{\varepsilon^2} \mathbf{1}_{\{
\mathcal{G} \omega_\varepsilon-\eta-\mu_\varepsilon>0
\}}\;\;\; a.e. \;\;\text{in}\;\; \Omega,
\]
 where $\mu_\varepsilon$ is a Lagrange multiplier so that  
$
\int_\Omega \omega_\varepsilon\;dx =1. 
$
The vorticity set $supp(\omega_\varepsilon)$ satisfies 
\[
supp(\omega_\varepsilon) \subset B_{r\varepsilon}(x_\varepsilon), 
\]
where $r>0$ is a constant and $x_\varepsilon\in \Omega$.  Moreover, as $\varepsilon\to 0$, 
 \[
 \begin{split}
&\text{dist}(x_\varepsilon, \mathcal{Q})\to 0,
\quad \mu_\varepsilon=\frac{1}{2\pi}\ln \frac{1}{\varepsilon}+O(1),
\end{split}
\]
where $\mathcal{Q}$ is defined in $(i)$ of Theorem~\ref{th3-1}. 
\end{theorem}

The  proof of Theorem~\ref{thm1-4} is based on finding  maximizers of kinetic energy 
\begin{equation}
\mathcal{E} (\omega)=\frac{1}{2}\int_{\Omega_L} \omega \mathcal{G} \omega \;dx-\int_{\Omega_L} \omega \eta \;dx
\end{equation}
over the following admissible set 
\begin{equation}
\mathcal{A}_\varepsilon:=\Big\{
\omega\in L^\infty(\Omega)\big| \int_\Omega \omega\; dx=1,\;\; 0\leq \omega\leq \frac{1}{\varepsilon^2},\;\; \text{supp} (\omega)\subset \Omega_L
\Big\},
\end{equation}
where 
$
\Omega_L=\{
x\in \Omega|\; -L<x_1<L
\}
$
for some large $L>0$. It is not difficult to find a maximizer $\omega_\varepsilon$ of $\mathcal{E}$ over $\mathcal{A}_\varepsilon$
and prove that it is a patch-type function (see Lemma~\ref{lem4-1}). To show $\omega_\varepsilon$  a weak solution of \eqref{eq1-8}, following Turkington's idea (see \cite{Tur83}) it needs to prove that the vorticity set $\text{supp} (\omega_\varepsilon)$ must be away from the boundary of $\Omega_L$, which is obtained by establishing  the asymptotic estimates of $\omega_\varepsilon$. Let us remark that the asymptotic estimates highly rely on the structure  of Green function $G(y,x)$, which is shown in Theorem~\ref{thm2-2}, and the locations of minimum point of $\mathcal{H}$, which is shown in Theorem~\ref{th3-1}.

When the velocity at infinity and circulation around the obstacle are large enough, we can find  vortex-patch solution whose vorticity concentrates near the obstacle $O_0$. 

\begin{theorem}\label{thm1-5} There exists a large constant $b_0>0$ such that for $b>b_0$ and $\Gamma/b \geq \pi \int_{\mathbb{R}\times \{x_2=\pi\}} | \frac{\partial \rho}{\partial \mathbf{n}} |\; dS_x$,  there exists $\varepsilon_0>0$,  depending on $b$ and $\Gamma$,  so that for $0<\varepsilon<\varepsilon_0$, \eqref{eq1-8} has a solution $\omega_\varepsilon$ with the form 
\[
\omega_\varepsilon=\frac{1}{\varepsilon^2} \mathbf{1}_{\{
\mathcal{G} \omega_\varepsilon-\eta-\mu_\varepsilon>0
\}}\;\;\; a.e. \;\;\text{in}\;\; \Omega,
\]
 where $\mu_\varepsilon$ is a Lagrange multiplier so that  
$
\int_\Omega \omega_\varepsilon\;dx =1. 
$
The vorticity set $supp(\omega_\varepsilon)$ satisfies 
\[
supp(\omega_\varepsilon) \subset B_{r\varepsilon}(x_\varepsilon), 
\]
where $r>0$ is a constant and $x_\varepsilon\in \Omega$.  Moreover, as $\varepsilon\to 0$, 
 \[
 \begin{split}
& \frac{\theta_1}{b+\lambda_{b,\Gamma}}<dist (x_\varepsilon,\partial O_0)< \frac{\theta_2}{b+\lambda_{b,\Gamma}},
\quad \mu_\varepsilon=\frac{1}{2\pi}\ln \frac{1}{\varepsilon}+O(1),
\end{split}
\]
where $\theta_1,\theta_2$ are two positive constants independent of $\varepsilon,b,\Gamma$, $\lambda_{b,\Gamma}$ is a positive constant depending on $b$ and $\Gamma$ (see \eqref{eq2-24}).. 

\end{theorem}

\begin{corollary} Under the assumptions in $(iii)$ of Theorem~\ref{th3-1}, expect the solution found in Theorem~\ref{thm1-5}, there exists another  vortex-patch solution $\omega_\varepsilon$ of \eqref{eq1-8}, whose vorticity concentrates near $\mathcal{Q}_2$ as $\varepsilon\to0$, namely away from the obstacle $O_0$. Here, $\mathcal{Q}_2$ is defined in  $(iii)$  of Theorem~\ref{th3-1}. 
\end{corollary}



This paper is organized as follows. In section 2, we will show the existence and basic estimates of  $G(y,x)$ and $\eta(x)$, and prove Theorem~\ref{thm2-2}, Theorem~\ref{th3-1}. In section 3, we will use a variational argument to prove Theorem~\ref{thm1-4} and theorem~\ref{thm1-5}. In appendix A, we will prove the existence and uniqueness of a harmonic problem, while in appendix B we will establish two Green's representation formulas.

\section{Existence and basic estimates for  $G(y,x)$ and $\eta(x)$}

In this section we will establish the existence and basic estimates of Green function $G(y,x)$ and background stream function $\eta(x)$, and hence prove Theorem~\ref{thm2-2} and Theorem~\ref{th3-1}.

\subsection{The Green function $G(y,x)$}

The purpose of this subsection is to show the existence of Green function of the prescribed circulation problem \eqref{eq2-1}, and further to establish the asymptotic behaviors for it, which play an important role in the proof of  our main results. 
Let us begin with the following lemma. 
\begin{lemma}\label{lem2-1y}
For any $f\in L^2(\Omega)$,   there exists a unique weak solution $u\in H^1(\Omega)$ of \eqref{eq2-1}. 
\end{lemma}

\begin{proof}Let  $f\in L^2(\Omega)$. We decompose $u=u_1+u_2$, where $u_1$ satisfies  
\begin{equation}\label{eq2-1y}
\begin{cases}
-\Delta u_1=f\quad \text{in}\; \Omega,\\
u_1=0\;\;\; \text{on}\; \partial \Omega,
\quad
u_1 \to 0\;\; \text{as}\;\; |x|\to \infty, 
\end{cases}
\end{equation}
and $u_2$ satisfies 
\begin{equation}\label{eq2-2y}
\begin{cases}
\Delta u_2=0\quad \text{in}\; \Omega,\\
u_2=\text{constant} \;\;  \text{on}\; \partial O_0,
\quad
u_2=0 \;\; \text{on}\; \partial S, 
\\
u_2 \to 0\;\; \text{as}\;\; |x|\to \infty, 
\\
\int_{\partial O_0} \frac{\partial u_2}{\partial \mathbf{n}}\; dS_x=-\int_{\partial O_0} \frac{\partial u_1}{\partial \mathbf{n}}\; dS_x.
\end{cases}
\end{equation}
It is easy to see that \eqref{eq2-1y} has a unique solution $u_1\in H_0^1(\Omega)$. Then the conclusion follows from Lemma~\ref{lemA-1}. 
\end{proof}

As a consequence of Lemma~\ref{lem2-1y},  one can define a Green operator $\mathcal{G}: L^2(\Omega) \to H^1(\Omega)$ such that  the weak solution $u\in H^1(\Omega)$ of \eqref{eq2-1} is equal to $\mathcal{G}f$, namely,  $u=\mathcal{G}f$. Moreover, we have 
\begin{lemma} \label{lem2-2y}
The Green operator $\mathcal{G}$ is continuous from $L^2(\Omega) $ to $H^1(\Omega)$. 
\end{lemma}

\begin{proof} Suppose that $f_n\to f_0$ in $L^2(\Omega)$ as $n\to \infty$. Let $u_n=\mathcal{G}f_n$ and  $u_0=\mathcal{G}f_0$. Then $u=u_n\in H^1(\Omega)$ (resp. $u=u_0$) is a weak solution of \eqref{eq2-1} with $f=f_n$ (resp. $f=f_0$). Multiplying  $-\Delta (u_n-u_0)=f_n-f_0$ by $u_n-u_0$ and integrating in $\Omega$, we get 
\[
\int_\Omega (-\Delta (u_n-u_0)) (u_n-u_0) dx=\int_\Omega (f_n-f_0) (u_n -u_0)dx.
\]
 Since $u_n, u_0=0$  on $\partial S$ and  $u_n, u_0=constant$  on $\partial O_0$, we  can extend  $u_n, u_0$ to $S$ such that $u_n, u_0\in H_0^1(S)$. So we have
\[
\begin{split}
\int_\Omega (-\Delta (u_n -u_0)) (u_n -u_0) dx=&\int_\Omega |\nabla (u_n -u_0)|^2 dx-\int_{\partial O_0} \frac{\partial (u_n -u_0)}{\partial \mathbf{n}} (u_n -u_0) dS_x
\\
=&\int_\Omega |\nabla (u_n -u_0)|^2 dx.
\end{split}
\]
By H\"older inequality and Poincar\'e inequality, we have 
\[
\begin{split}
\int_\Omega |\nabla (u_n -u_0)|^2 dx=& \int_\Omega (f_n-f_0) (u_n -u_0) dx \leq \|f_n-f_0\|_{L^2(\Omega)} \|u_n -u_0\|_{L^2(\Omega)} 
 \\
 \leq& C \|f_n-f_0\|_{L^2(\Omega)} \|\nabla(u_n -u_0)\|_{L^2(\Omega)},
 \end{split}
\]
which, together with $f_n\to f_0$ in $L^2(\Omega)$, gives 
\[
\|\nabla(u_n -u_0)\|_{L^2(\Omega)}\to 0\;\;\;\text{as}\;\;n\to \infty.
\]
Using  Poincar\'e inequality again, we obtain that as $n\to \infty$, $u_n\to u_0$ in  $H^1(\Omega)$.
\end{proof}

We are able to prove Theorem~\ref{thm2-2}. 

\begin{proof}[{\bf Proof of Theorem~\ref{thm2-2} }]
Since  $C_0^\infty(\Omega)$ is density in $L^2(\Omega)$, by Lemma~\ref{lem2-2y} we only need to show \eqref{eq2-2} in the case $f\in C_0^\infty(\Omega)$. 

Let $f\in C_0^\infty(\Omega)$ and define 
\[
I[f](x)=\int_{\Omega} \big(G_0(y,x)+\lambda_0 \rho(y)\rho(x)\big) f(y)\; dy, \;\;\; x\in \Omega.
\]
Our goal is to show that $\mathcal{G}f=I[f]$  in $\Omega$. Let $u_0(x)=\int_{\Omega} G_0(y,x) f(y)\; dy$. Then $u_0\in C^2(\bar{\Omega})\cap H_0^1(\Omega)$  satisfies 
\begin{equation}\label{eq2-5}
-\Delta u_0=f\quad \text{in}\; \Omega.
\end{equation}
By \eqref{eq2-3} and \eqref{eq2-5}, we find that 
\[
\int_\Omega f \rho \;dx=\int_{\Omega} (-\Delta u_0)\rho+(\Delta \rho) u_0\;dx=\int_{\partial \Omega} -\frac{\partial u_0}{\partial \mathbf{n}}\rho+\frac{\partial \rho}{\partial \mathbf{n}}  u_0\;dS_x
=-\int_{\partial \Omega} \frac{\partial u_0}{\partial \mathbf{n}}\;dS_x.
\]
As a result, we have 
\[
\int_{\partial O_0}\frac{\partial I[f]}{\partial \mathbf{n}}\;dS_x=\int_{\partial O_0}\frac{\partial u_0}{\partial \mathbf{n}}\;dS_x+\lambda_0\big(\int_\Omega f \rho\; dx\big)\int_{\partial O_0} \frac{\partial \rho}{\partial \mathbf{n}}\;dS_x=0.
\]
Then, it is easy to see that  $I[f]$ is a solution of \eqref{eq2-1}.  By Lemma~\ref{lem2-1y}, $I[f]$ is the unique solution of \eqref{eq2-1}. Thus, $\mathcal{G}f=I[f]$ in $\Omega$.

By the maximum principle, we have  $0<\rho<1$ in $\Omega$.
Then the Hopf lemma implies that $\frac{\partial \rho}{\partial \mathbf{n}}$ is positive on $\partial O_0$. So, 
$
\lambda_0=\Big(\int_{\partial O_0}\; \frac{\partial \rho}{\partial \mathbf{n}}\; dS_x\Big)^{-1}>0.
$
We complete the proof. 
\end{proof}

We  begin to study the asymptotic behaviors of $G(y,x)$ at infinity via some suitable Green representation formulas relying on $G_S(y,x)$, where $G_S(y,x)$ is given in \eqref{eq1-18}. 
To this purpose, we first introduce  the asymptotic behaviors of $G_S(y,x)$ at infinity. 
\begin{lemma}\label{lem2-4} For $x,y\in S$ with $|y_1|<L_0$ and large $|x|$, where $L_0>0$ is a  constant, 
\begin{equation}\label{eq2-8}
	G_S(y,x)=\frac{\sin x_2\sin y_2}{\pi e^{|x_1-y_1|}}+O\left(\frac{\sin x_2}{e^{2|x_1|}}\right).
	\end{equation}
\end{lemma}

\begin{proof} Since 
\[
G_S(y,x)=\frac{1}{4\pi}\ln \left(1+\frac{2 \sin x_2\sin y_2}{\cosh(y_1-x_1)-\cos (y_2-x_2)}\right), 
\]
the estimate \eqref{eq2-8}  follows from the Taylor expansion.  
\end{proof}


We first estimate $\rho$ at infinity, where $\rho$ is defined in \eqref{eq2-3}. 

\begin{lemma}\label{lem2-5} For $x\in \Omega$ with large $|x|$, 
\begin{equation}\label{eq2-12}
\begin{split}
\rho(x)=\begin{cases}\rho_+\frac{\sin x_2}{e^{|x_1|}}+O\left(\frac{\sin x_2}{e^{2|x_1|}}\right),\quad &\text{if}\;\; x_1>0,\\
\rho_-\frac{\sin x_2}{e^{|x_1|}}+O\left(\frac{\sin x_2}{e^{2|x_1|}}\right),\quad &\text{if}\;\; x_1<0,
\end{cases}
\end{split}
\end{equation}
where $\rho_+=\frac{1}{\pi }\int_{\partial O_0} e^{y_1}\sin y_2 \frac{\partial \rho (y)}{\partial \mathbf{n}}\;dS_y$,  $\rho_-=\frac{1}{\pi }\int_{\partial O_0} e^{-y_1}\sin y_2 \frac{\partial \rho (y)}{\partial \mathbf{n}}\;dS_y$ are two positive constants. 

\end{lemma}

\begin{proof}
 By \eqref{eqC-1} in Lemma~\ref{lemC-1}, we have that for $x\in \Omega$,
\[
\rho(x)=\int_{\partial O_0}\,G_S(y,x)\frac{\partial \rho (y)}{\partial \mathbf{n}}-\rho(y)\frac{\partial G_S(y,x)}{\partial \mathbf{n}}\, dS_y=\int_{\partial O_0}\,G_S(y,x)\frac{\partial \rho (y)}{\partial \mathbf{n}}\, dS_y,
\]
where we have used 
\[
\int_{\partial O_0}\,\rho(y)\frac{\partial G_S(y,x)}{\partial \mathbf{n}}\, dS_y=\int_{\partial O_0}\,\frac{\partial G_S(y,x)}{\partial \mathbf{n}}\, dS_y=-\int_{O_0} \Delta_y G_S(y,x)\;dy=0.
\]
Then \eqref{eq2-12} follows from \eqref{eq2-8}.  Since  $\frac{\partial \rho}{\partial \mathbf{n}}>0$,  $e^{y_1}\sin y_2>0$ and $ e^{-y_1}\sin y_2>0$ on $\partial O_0$,  we have  
$\rho_+,\rho_->0$. 
\end{proof}

We now show the Green function $G$ at infinity. 

\begin{lemma}\label{lem2-6}For any $x,z\in \Omega$ with large $|z|$,
\begin{equation}\label{eq2-14}
\begin{split}
G(z,x)-G_S(z,x)=\begin{cases}
c_+(x)\frac{\sin z_2}{e^{|z_1|}}+O\left(\frac{\sin z_2}{e^{2|z_1|}}\right),\;\;&\text{if}\;\; z_1>0,
\\
c_-(x)\frac{\sin z_2}{e^{|z_1|}}+O\left(\frac{\sin z_2}{e^{2|z_1|}}\right),\;\;&\text{if}\;\; z_1<0,
\end{cases}
\end{split}
\end{equation}
where $c_+(x)=\frac{1}{\pi}\int_{\partial O_0} e^{y_1}\sin y_2  \frac{\partial G(y,x)}{\partial \mathbf{n}}dS_y$ and  $c_-(x)=\frac{1}{\pi}\int_{\partial O_0} e^{-y_1}\sin y_2  \frac{\partial G(y,x)}{\partial \mathbf{n}}dS_y$, satisfying 
\begin{equation}\label{eq2-15}
|c_+(x)|,\;|c_-(x)|\leq C\rho(x). 
\end{equation}

\end{lemma}

\begin{proof} Fix $x\in \Omega$. Set $f(y)=G(y,x)-G_S(y,x)$. It satisfies the conditions  in Lemma~\ref{lemC-1}. So by \eqref{eq2-4}, for $z\in \Omega$, 
\[
\begin{split}
f(z)=&\int_{\partial O_0}\,G_S(y,z)\frac{\partial f(y)}{\partial \mathbf{n}}-f(y)\frac{\partial G_S(y,z)}{\partial \mathbf{n}}\, dS_y
\\
=&\int_{\partial O_0}\,G_S(y,z)\frac{\partial \big(G(y,x)-G_S(y,x)\big)}{\partial \mathbf{n}}-\big(G(y,x)-G_S(y,x)\big)\frac{\partial G_S(y,z)}{\partial \mathbf{n}}\, dS_y
\\
=&\int_{\partial O_0}\,G_S(y,z)\frac{\partial G(y,x)}{\partial \mathbf{n}} \, dS_y -\rho(x) \int_{\partial O_0}\, \frac{\partial G_S(y,z)}{\partial \mathbf{n}}\, dS_y 
\\
&+\int_{\partial O_0}\, G_S(y,x)\frac{\partial G_S(y,z)}{\partial \mathbf{n}}-G_S(y,z)\frac{\partial G_S(y,x)}{\partial \mathbf{n}} \, dS_y.
\end{split}
\]
Since $x,z\notin O_0$, $G_S(\cdot,x)$ and $G_S(\cdot,z)$ are harmonic in $O_0$. So we get 
\[
\begin{split}
\int_{\partial O_0}\, \frac{\partial G_S(y,z)}{\partial \mathbf{n}}\, dS_y=-\int_{O_0}\; \Delta_y G_S(y,z) dy=0,
\end{split}
\]
and 
\[
\begin{split}
\int_{\partial O_0}\, G_S(y,x)&\frac{\partial G_S(y,z)}{\partial \mathbf{n}}-G_S(y,z)\frac{\partial G_S(y,x)}{\partial \mathbf{n}} \, dS_y
\\
=&-\int_{O_0}\; G_S(y,x) \Delta_y G_S(y,z)-G_S(y,z) \Delta_y G_S(y,x)\; dy=0.
\end{split}
\]
Thus, 
\begin{equation}\label{eq2-17}
f(z)=\int_{\partial O_0}\,G_S(y,z)\frac{\partial G(y,x)}{\partial \mathbf{n}} \, dS_y,\quad z\in \Omega. 
\end{equation}

By Lemma~\ref{lemC-1}, we have 
$
\rho(x)=-\int_{\partial O_0}\; \frac{\partial G_0(y,x)}{\partial \mathbf{n}} \, dS_y
$
for $x\in \Omega$. We also have $\frac{\partial G_0(y,x)}{\partial \mathbf{n}} <0$ on $\partial O_0$ via the Hopf lemma.  So  for $x\in \Omega$, 
\[
\rho(x)=\int_{\partial O_0}\,\Big|\frac{\partial G_0(y,x)}{\partial \mathbf{n}} \Big|\, dS_y,\;\;x\in \Omega.
\]
Since $0<\rho(x)\leq 1$ for $x\in \bar{\Omega}$, we obtain from \eqref{eq2-4}  that 
\[
\begin{split}
\int_{\partial O_0}\,\Big|\frac{\partial G(y,x)}{\partial \mathbf{n}} \Big|\, dS_y=\int_{\partial O_0}\,\Big|\frac{\partial \big(G_0(y,x) +\lambda_0\rho(x)\rho(y)      \big)}{\partial \mathbf{n}} \Big|\, dS_y
\\
\leq \int_{\partial O_0}\,\Big|\frac{\partial G_0(y,x)}{\partial \mathbf{n}} \Big|\, dS_y+\rho(x)\leq 2\rho(x)\leq 2.
\end{split}
\]
Hence, \eqref{eq2-14} follows from \eqref{eq2-8} amd \eqref{eq2-17}.

Next we show \eqref{eq2-15}.  By \eqref{eq2-4}, we have 
\[
\begin{split}
|c_+(x)|=&\Big|\frac{1}{\pi}\int_{\partial O_0} e^{y_1}\sin y_2  \frac{\partial G(y,x)}{\partial \mathbf{n}}dS_y\Big|
\\
=&\Big|
\frac{1}{\pi}\int_{\partial O_0} e^{y_1}\sin y_2  \frac{\partial G_0(y,x)}{\partial \mathbf{n}}dS_y+\lambda_0 \rho_+ \rho(x)
\Big|
\\
\leq& C\int_{\partial O_0}\,\Big|\frac{\partial G_0(y,x)}{\partial \mathbf{n}} \Big|\, dS_y+\lambda_0 \rho_+ \rho(x)
\\
\leq& C\rho(x),
\end{split}
\]
where $\rho_+$ is given in \eqref{eq2-12}.
Similarly, we can prove $|c_-(x)|\leq C\rho(x)$. 
\end{proof}

Based on Lemma~\ref{lem2-6}, we are able to prove the asymptotic behaviors of $H(x,x)$ at infinity, where $H(y,x)$ is the regular part of $G(y,x)$.

\begin{proposition}\label{pro2-7}For $x\in \Omega$ with large $|x|$, 
\begin{equation}\label{eq2-19}
H(x,x)=-\frac{1}{2\pi}\ln (2\sin x_2)+O\left(\frac{\sin x_2}{e^{2|x_1|}}\right).
\end{equation}
\end{proposition}

\begin{proof}
Let $H_S(y,x)$ be the regular part of $G_S(y,x)$ such that $G_S(y,x)=\frac{1}{2\pi}\ln \frac{1}{|y-x|}-H_S(y,x)$. Then, $G(z,x)-G_S(z,x)=H_S(z,x)-H(z,x)$. 
By \eqref{eq2-14}-\eqref{eq2-15} and taking $z=x$ in \eqref{eq2-14}, we get that for  $x\in \Omega$ with large $|x|$,
\[
H(x,x)=H_S(x,x)+O\left( \rho(x)\frac{\sin x_2}{e^{|x_1|}}+\frac{\sin x_2}{e^{2|x_1|}}  \right).
\]
Then \eqref{eq2-19} follows from \eqref{eq2-12} and 
	\begin{equation}\label{eq2-7}
	H_S(x,x)=-\frac{1}{2\pi}\ln (2\sin x_2),\;\;\; x\in S. 
	\end{equation}
\end{proof}

We also have the following boundary estimate for $H(x,x)$. 

\begin{proposition}\label{pro2-8} As $d:=\text{dist}(x,\partial \Omega)\to 0$,
\begin{equation}\label{eq2-21}
H(x,x)=\frac{1}{2\pi}\ln \frac{1}{2d}+O(1).
\end{equation}
\end{proposition}

\begin{proof} By \eqref{eq2-4}, 
\[
G(y,x)-G_0(y,x)=\lambda_0\rho(x)\rho(y), \; \text{i.e.},\; H_0(y,x)-H(y,x)=\lambda_0\rho(x)\rho(y),
\]
where $G_0(y,x)$ is the Green function of $-\Delta$ in $\Omega$ with zero Dirichlet boundary condition, 
$H_0(y,x)$ is the regular part of $G_0(y,x)$.   Since  $0<\rho(x)\leq 1$ for $x\in \bar{\Omega}$, by letting $y=x$ in above equality we have 
\begin{equation}\label{eq2-18y}
H(x,x)=H_0(x,x)-\lambda_0\rho^2(x)=H_0(x,x)+O(1).
\end{equation}

 By \eqref{eq2-8}, there exists a large constant $L>0$  such that   for    $x\in \Omega$ with $|x_1|>L$,
 \[
 \|G_S(\cdot,x)\|_{L^\infty (\partial O_0)} \leq \sin x_2.
\]
Let $R(y,x)=G_0(y,x)-G_S(y,x)=H_S(y,x)-H_0(y,x)$. Then  $R(\cdot,x)$ is harmonic in $\Omega$, $R(y,x)=-G_S(y,x)$ for $y\in \partial O_0$, $R(y,x)=0$ for $y\in \partial  S$, $R(y,x)\to 0$ as $|y|\to \infty$. 
The maximum principle implies that 
\[
|R(y,x)|\leq  \|G_S(\cdot,x)\|_{L^\infty (\partial O_0)} \leq \sin x_2,\;\; \forall y\in \Omega. 
\]
Taking  $y=x$, we get 
\[
|R(x,x)|\leq \sin x_2. 
\]
Then we have 
\[
H_0(x,x)=H_S(x,x)-R(x,x)=\frac{1}{2\pi}\ln \frac{1}{2\sin x_2}+O(\sin x_2)
\]
Since  $\sin x_2=d+O(d^3)$, we have that  for $x\in \Omega$ with $|x_1|>L$, 
\begin{equation}\label{eq2-19y}
H_0(x,x)=\frac{1}{2\pi }\ln \frac{1}{2d}+O(d)\quad \text{as}\;\; d\to 0. 
\end{equation}

For $x\in \Omega$ with $|x_1|\leq L$, we can use a similar argument, for the asymptotic behavior of  Robin function in bounded domain, as in \cite[Lemma A.1]{HYY} to show that 
\begin{equation}\label{eq2-20y}
H_0(x,x)=\frac{1}{2\pi }\ln \frac{1}{2d}+O(d)\quad \text{as}\;\; d\to 0. 
\end{equation}

Hence, \eqref{eq2-21} follows from \eqref{eq2-18y}-\eqref{eq2-20y}. 
\end{proof}


\subsection{The background streamfunction $\eta$}

In this subsection we study  the existence, uniqueness and asymptotic behaviors of irrotational background flows governed by the  system \eqref{eq2-1} with respect to the stream function $\eta$. 
To solve this system, we first decompose $\eta$ as follows 
\[
\eta(x)=b(x_2+\beta(x)). 
\]
Then $\beta$ satisfies  the following equation
\begin{equation}\label{eq2-23}
\begin{cases}
\Delta \beta=0\quad \text{in}\; \Omega,\\
\beta=-x_2-\lambda/b \;\;  \text{on}\; \partial O_0,\quad \beta=0 \;\; \text{on}\;\; \partial S, 
\\
\nabla \beta \to (0,0)\;\; \text{as}\;\; |x|\to \infty, 
\\
\int_{\partial O_0} \frac{\partial \beta}{\partial \mathbf{n}} =-\Gamma/b, 
\end{cases}
\end{equation}
where we have used 
\[
\int_{\partial O_0} \frac{\partial x_2}{\partial \mathbf{n}} dS_x=-\int_{O_0} \Delta x_2 dx=0. 
\]

By Lemma~\ref{lemA-1}, we have 
\begin{theorem}\label{th2-9} Let $\Gamma$ and $b$ be two prescribed non-zero constants. Then \eqref{eq2-23} has a unique solution $(\beta,\lambda_{b,\Gamma})$ in $H^1(\Omega)\times \mathbb{R}$.  Moreover, $\beta\in C^2(\bar{\Omega})$ satisfies that as $|x|\to \infty$,   $\beta\to 0$.
\end{theorem}

\begin{lemma}\label{lem2-10}The flux constant $\lambda_{b,\Gamma}$ has the following form
\begin{equation}\label{eq2-24}
\lambda_{b,\Gamma}=\frac{\Gamma-b\pi \int_{\mathbb{R}\times \{x_2=\pi\}} \big| \frac{\partial \rho}{\partial \mathbf{n}} \big|\; dS_x  }{ \int_{\partial O_0} \frac{\partial \rho}{\partial \mathbf{n}}\; dS_x}.
\end{equation}
\end{lemma}

\begin{proof}By the Green's formula, we have 
\[
\begin{split}
0=&\int_\Omega \rho \Delta(\beta+x_2)-(\beta+x_2)\Delta \rho \;dx
\\
=&\int_{\partial \Omega}  \rho \frac{\partial ( \beta+x_2)}{\partial \mathbf{n}}- ( \beta+x_2) \frac{\partial \rho}{\partial \mathbf{n}}\; dS_x
\\
=&\int_{\partial O_0}  \frac{\partial ( \beta+x_2)}{\partial \mathbf{n}}\; dS_x
-\pi\int_{\mathbb{R}\times \{x_2=\pi\}}  \frac{\partial \rho}{\partial \mathbf{n}}\; dS_x+\frac{\lambda_{b,\Gamma}}{b}\int_{\partial O_0} \frac{\partial \rho}{\partial \mathbf{n}}\; dS_x
\\
=&-\frac{\Gamma}{b}
-\pi\int_{\mathbb{R}\times \{x_2=\pi\}}  \frac{\partial \rho}{\partial \mathbf{n}}\; dS_x+\frac{\lambda_{b,\Gamma}}{b}\int_{\partial O_0} \frac{\partial \rho}{\partial \mathbf{n}}\; dS_x.
\end{split}
\]
By the maximum principle and Hopf lemma, we can show that $\int_{\mathbb{R}\times \{x_2=\pi\}}  \frac{\partial \rho}{\partial \mathbf{n}}\; dS_x=-\int_{\mathbb{R}\times \{x_2=\pi\}} | \frac{\partial \rho}{\partial \mathbf{n}} |\; dS_x$.  So, we have proved \eqref{eq2-24}.
\end{proof}


The  asymptotic behaviors  of $\beta$ at infinity are stated as follows. 

\begin{lemma}\label{lem2-11} Let  $b$  and $\Gamma$  be two positive constants, such that $\Gamma/b \geq \pi \int_{\mathbb{R}\times \{x_2=\pi\}} | \frac{\partial \rho}{\partial \mathbf{n}} |\; dS_x$.  Then 
for $x\in \Omega$ and with large $|x|$,
\begin{equation}\label{eq2-25}
\beta(x)=\begin{cases}-\beta_+\frac{\sin x_2}{e^{|x_1|}}+O\left(\frac{\Gamma}{b}\frac{\sin x_2}{e^{2|x_1|}}\right),\quad &\text{if}\;\; x_1>0,\\
-\beta_-\frac{\sin x_2}{e^{|x_1|}}+O\left(\frac{\Gamma}{b}\frac{\sin x_2}{e^{2|x_1|}}\right),\quad &\text{if}\;\; x_1<0,
\end{cases}
\end{equation}
where 
$\beta_{+}, \beta_{-}$ are two positive constants:
\[
\beta_+=-\frac{1}{\pi} \int_{\partial O_0} e^{y_1}\sin y_2  \frac{\partial ( \beta(y)+y_2)}{\partial \mathbf{n}}\, dS_y,\;\;\; \beta_-=-\frac{1}{\pi} \int_{\partial O_0} e^{-y_1}\sin y_2  \frac{\partial ( \beta(y)+y_2)}{\partial \mathbf{n}}\, dS_y.
\]

Moreover,  there exists two positive constants $C^\prime$ and $C^{\prime\prime}$,  depending only on $\Omega$,   such that 
\begin{equation}\label{eq2-27}
C^\prime \Gamma/b \leq \beta_{+}, \beta_{-} \leq C^{\prime\prime} \Gamma/b.
\end{equation}
 \end{lemma}

\begin{proof}
Since $\Gamma/b \geq \pi \int_{\mathbb{R}\times \{x_2=\pi\}} | \frac{\partial \rho}{\partial \mathbf{n}} |\; dS_x=- \pi \int_{\mathbb{R}\times \{x_2=\pi\}}  \frac{\partial \rho}{\partial \mathbf{n}}\; dS_x$,  by \eqref{eq2-24} we have $\lambda_{b,\Gamma}/b\geq 0$ in view of $\int_{\partial O_0} \frac{\partial \rho}{\partial \mathbf{n}}\; dS_x>0$. By the maximum principle, 
\[
-\lambda_{b,\Gamma}/b<\beta+x_2<\pi \;\;\;\text{in}\;\;\Omega.
\]
Since $\beta+x_2=-\lambda_{b,\Gamma}/b$ on $\partial O_0$, by the Hopf lemma we have $\frac {\partial(\beta+x_2)}{\partial \mathbf{n}}<0$ on $\partial O_0$.  So we get 
\begin{equation}\label{eq2-28}
\int_{\partial O_0}\big| \frac {\partial(\beta+x_2)}{\partial \mathbf{n}} \big|dS_x=-\int_{\partial O_0} \frac {\partial(\beta+x_2)}{\partial \mathbf{n}} dS_x=\frac{\Gamma}{b}. 
\end{equation}
By Lemma~\ref{lemC-1},  for $x\in \Omega$,
\[
\begin{split}
\beta(x)=&\int_{\partial O_0}\,G_S(y,x)\frac{\partial \beta(y)}{\partial \mathbf{n}}-\beta(y)\frac{\partial G_S(y,x)}{\partial \mathbf{n}}\, dS_y
\\
=&\int_{\partial O_0}\,G_S(y,x)\frac{\partial \beta(y)}{\partial \mathbf{n}}-\big(-y_2-\lambda_{b,\Gamma}/b\big)\frac{\partial G_S(y,x)}{\partial \mathbf{n}}\, dS_y
\\
=&\int_{\partial O_0}\,G_S(y,x)\frac{\partial \beta(y)}{\partial \mathbf{n}}-G_S(y,x)\frac{\partial    \big(-y_2-\lambda_\beta/b\big) }{\partial \mathbf{n}}\, dS_y
\\
=&\int_{\partial O_0}\,G_S(y,x)\frac{\partial ( \beta(y)+y_2)}{\partial \mathbf{n}}\, dS_y.
\end{split}
\]
Then, \eqref{eq2-25} follows from \eqref{eq2-8}.


Note that there exist two positive $C^{\prime}$ and $C^{\prime\prime}$ such that for every $y\in \partial O_0$, 
\[
C^{\prime}\leq \frac{1}{\pi}e^{-y_1}\sin y_2, \;\frac{1}{\pi}e^{y_1}\sin y_2 \leq C^{\prime\prime}. 
\]
Then, \eqref{eq2-27} follows from \eqref{eq2-28}.
\end{proof}

We are able to prove the existence of critical points of the following Kirchhoff-Routh function
\[
\mathcal{H}(x)=H(x,x)+2b(x_2+\beta(x)).
\]

\begin{proof}[{\bf Proof of Theorem~\ref{th3-1}}]



{\bf Proof of $(i)$. }

By \eqref{eq2-19} and \eqref{eq2-25}, we have  that for $x\in \Omega$ with large $|x|$, 
\begin{equation}\label{eq3-1}
\mathcal{H}(x)=\begin{cases}
 -\frac{1}{2\pi}\ln (2\sin x_2)
+2b\Big(x_2-\beta_+\frac{ \sin x_2}{e^{x_1}}\Big)
+O\left((1+\Gamma)\frac{\sin x_2}{e^{2|x_1|}}\right),  &x_1>0,
\\
 -\frac{1}{2\pi}\ln (2\sin x_2)
+2b\Big(x_2-\beta_-\frac{ \sin x_2}{e^{-x_1}}\Big)
+O\left((1+\Gamma)\frac{\sin x_2}{e^{2|x_1|}}\right),  &x_1<0.
\end{cases}
\end{equation}
Let  $\Omega_L:=\{x\in \Omega|\; -L<x_1<L\}$.  Without lost of generality, we may assume that $\beta_+\geq \beta_->0$. Then by \eqref{eq3-1} and Proposition~\ref{pro2-8},  there exists $x^L\in \{L\}\times (0,\pi)$ such that 
\[
\mathcal{H}(x^L)=\inf_{x\in \partial \Omega_L} \mathcal{H}(x)=-\frac{1}{2\pi}\ln (2\sin t_0)
+2b\Big(t_0-\frac{\beta_+}{e^{L}}\sin t_0\Big)
+O\big((1+\Gamma)e^{-2L}\big),
\]
where $t_0\in (0,\frac{\pi}{2})$ satisfies $\cot t_0= 4b \pi$. 
Replacing $L$ by $L/2$, we can find  $x^{L/2} \in \{L/2\}\times (0,\pi)$ such that 
\begin{equation}\label{eq2-26Y}
\mathcal{H}(x^{L/2})=-\frac{1}{2\pi}\ln (2\sin t_0)
+2b\Big(t_0-\frac{\beta_+}{e^{L/2}}\sin t_0\Big)
+O((1+\Gamma)e^{-L}).
\end{equation}
So for $L>0$ large enough, we have 
$
\mathcal{H}(x^{L/2})<\inf_{x\in \partial \Omega_L} \mathcal{H}(x).
$
Since $-\frac{\beta_+}{e^{L}}\sin t_0$ is increasing with respect to $L$,  one can further find that 
\[
\mathcal{H}(x^{L/2})<\inf_{x\in \Omega\setminus \Omega_L} \mathcal{H}(x).
\]
This shows that $\mathcal{Q}$ is  non-empty  and $\mathcal{Q} \subset \Omega_L$. By Proposition~\ref{pro2-8} again, we have 
\[
\text{dist}(\mathcal{Q},\partial \Omega)>0.
\]
Then, $(i)$ follows. 

{\bf Proof of $(ii)$.} 

  Since $\beta(x)$ satisfies \eqref{eq2-23}, it easy to see that  $\beta(x)$ has the following expansion
\[
\beta(x)=-\xi(x)-\frac{\lambda_{b,\Gamma}}{b}\rho(x)
\]
where $\lambda_{b,\Gamma}$ is given in \eqref{eq2-24}, $\rho(x)$ is defined in \eqref{eq2-3} and $\xi(x)$ is a unique solution of 
\[
\begin{cases}
\Delta \xi=0\quad \text{in}\; \Omega,\\
\xi=x_2 \;\;  \text{on}\;  \partial O_0,
\;\;\;
\xi=0 \;\; \text{on}\;\;  \partial S, 
\\
\xi \to 0\;\; \text{as}\;\; |x|\to \infty.
\end{cases}
\]
Then $\eta$ has the expansion
\begin{equation}
\eta(x)=b\big(
x_2-\xi(x)-\frac{\lambda_{b,\Gamma}}{b}\rho(x)
\big).
\end{equation}
For $x$ near $\partial O_0$, define $d=dist (x,\partial O_0)$. Let $\hat{x}\in \partial O_0$ be the unique point so that $|x-\hat{x}|=d$. 
Using the Tayler expansions, we have
\[
x_2-\xi(x)=-\frac{\partial (\hat{x}_2-\xi(\hat{x}))}{\partial \mathbf{n}} d+O(d^2),\quad \rho(x)=1-\frac{\partial \rho(\hat{x})}{\partial \mathbf{n}}d+O(d^2). 
\]
So, 
\begin{equation}\label{eq4-2y}
\eta(x)=-\lambda_{b,\Gamma}+A_{b,\Gamma}(\hat{x}) d +O\big((b+\lambda_{b,\Gamma})d^2\big),
\end{equation}
where 
\[
A_{b,\Gamma}(\hat{x})=b\frac{\partial (\xi(\hat{x})-\hat{x}_2)}{\partial \mathbf{n}} +\lambda_{b,\Gamma}\frac{\partial \rho(\hat{x})}{\partial \mathbf{n}}.
\]
By the Hopf lemma, we have $\frac{\partial (\xi(\hat{x})-\hat{x}_2)}{\partial \mathbf{n}} >0$ and  $\frac{\partial \rho(\hat{x})}{\partial \mathbf{n}}>0$ for all $\hat{x}\in \partial O_0$. Since $ \partial O_0$ is a compact set, there exists $0<c_1<c_2$ such that 
\[
c_1<\frac{\partial (\xi(\hat{x})-\hat{x}_2)}{\partial \mathbf{n}}<c_2,\;\;\; c_1< \frac{\partial \rho(\hat{x})}{\partial \mathbf{n}}<c_2.
\]
Since   $\Gamma/b \geq \pi \int_{\mathbb{R}\times \{x_2=\pi\}} | \frac{\partial \rho}{\partial \mathbf{n}} |\; dS_x$, we have $\lambda_{b,\Gamma}\geq 0$. 
So for each $\hat{x}\in \partial O_0$, 
\begin{equation}\label{eq4-3y}
c_1(b+\lambda_{b,\Gamma})<A_{b,\Gamma}(\hat{x})<c_2(b+\lambda_{b,\Gamma}). 
\end{equation}

By Proposition~\ref{pro2-8}, we get from \eqref{eq4-2y} that there exists a small constant $\delta>0$ such that for $0<d<\delta$,
\begin{equation}\label{eq4-4y}
\mathcal{H}(x)=\frac{1}{2\pi}\ln \frac{1}{2 d}+2A_{b,\Gamma}(\hat{x}) d -2\lambda_{b,\Gamma}+O\big(d+(b+\lambda_{b,\Gamma})d^2\big).
\end{equation}
Let $0<\theta_1<1/2$ be a small constant and $\theta_2>2$ be a large constant, which will be determined later. Define  
\[
D_{b,\Gamma}:=\Big\{ x\in \Omega:\; \frac{\theta_1}{b+\lambda_{b,\Gamma}}<d< \frac{\theta_2}{b+\lambda_{b,\Gamma}}
\Big\}.
\]
For any $x\in \bar{D}_{b,\Gamma}$, $d=\frac{\theta}{b+\lambda_{b,\Gamma}}$ for some $\theta_1\leq \theta\leq \theta_2$. So we have 
\begin{equation}\label{eq4-5y}
\begin{split}
\mathcal{H}(x)=&\frac{1}{2\pi}\ln \frac{b+\lambda_{b,\Gamma}}{2 \theta}+2A_{b,\Gamma}(\hat{x}) \frac{\theta}{b+\lambda_{b,\Gamma}} -2\lambda_{b,\Gamma}+O\Big(\frac{1}{b+\lambda_{b,\Gamma}}\Big),
\end{split}
\end{equation}
which combining \eqref{eq4-3y} gives that for $x\in \partial D_{b,\Gamma}$, namely $\theta=\theta_1$ or $\theta=\theta_2$, 
\begin{equation}\label{eq4-6y}
\mathcal{H}(x)>\frac{1}{2\pi}\ln (b+\lambda_{b,\Gamma})-2\lambda_{b,\Gamma}+\min_{i=1,2}\Big\{\frac{1}{2\pi}\ln \frac{1}{2\theta_i}+2c_1\theta_i\Big\}+O\Big(\frac{1}{b+\lambda_{b,\Gamma}}\Big),
\end{equation}
if $b>0$ large enough.

On the other hand, take some $x\in D_{b,\Gamma}$ with $d=\frac{1}{b+\lambda_{b,\Gamma}}$.  By \eqref{eq4-3y} and \eqref{eq4-5y}, we have 
\begin{equation}\label{eq4-7y}
\begin{split}
\mathcal{H}(x)<&\frac{1}{2\pi}\ln (b+\lambda_{b,\Gamma})-2\lambda_{b,\Gamma}+\frac{1}{2\pi}\ln \frac{1}{2}+2c_2+O\Big(\frac{1}{b+\lambda_{b,\Gamma}}\Big),
\end{split}
\end{equation}
if $b>0$ is large enough. 
By taking $\theta_1>0$ small enough and $\theta_2>0$ large enough, one can prove form \eqref{eq4-6y} and \eqref{eq4-7y} that 
\[
\inf_{x\in D_{b,\Gamma}} \mathcal{H}(x)<\inf_{x\in \partial D_{b,\Gamma}} \mathcal{H}(x),
\]
and $(ii)$ follows.

{\bf Proof of $(iii)$.} 

The existence of $\mathcal{Q}_1$ follows from $(ii)$. We next prove the existence of $\mathcal{Q}_2$. 

Let $\sigma>0$ be a small constant. Then  for  
\[
  \pi \int_{\mathbb{R}\times \{x_2=\pi\}} | \frac{\partial \rho}{\partial \mathbf{n}} |\; dS_x\leq \Gamma/b\leq  \pi \int_{\mathbb{R}\times \{x_2=\pi\}} | \frac{\partial \rho}{\partial \mathbf{n}} |\; dS_x+\sigma \int_{\partial O_0} \frac{\partial \rho}{\partial \mathbf{n}} \; dS_x,
  \]
we get from \eqref{eq2-24} that $0\leq \lambda_{b,\Gamma} \leq \sigma b$. So by \eqref{eq4-3y}-\eqref{eq4-4y}, for $x\in \Omega$ with $d=\delta$, 
  \begin{equation}\label{eq2-34y}
\begin{split}
\mathcal{H}(x)=&\frac{1}{2\pi}\ln \frac{1}{2 \delta}+2A_{b,\Gamma}(\hat{x}) \delta -2\lambda_{b,\Gamma}+O\big(\delta+(b+\lambda_{b,\Gamma})\delta^2\big)
\\
>&\frac{1}{2\pi}\ln \frac{1}{2 \delta}+2c_1(b+\lambda_{b,\Gamma}) \delta -2\lambda_{b,\Gamma}+O\big(\delta+(b+\lambda_{b,\Gamma})\delta^2\big)
\\
>&\frac{1}{2\pi}\ln \frac{1}{2 \delta}+(c_1\delta-2\sigma)b,
\end{split}
\end{equation}
if $b>0$ is large enough. 

On the other hand, by \eqref{eq2-26Y} there exists $x^{L/2}\in \Omega_\delta$ such that 
\[
\mathcal{H}(x^{L/2})<-\frac{1}{2\pi}\ln (2\sin t_0)+2bt_0
\]
if $L>0$ is large enough, where $\Omega_\delta:=\{x\in \Omega:\; \text{dist}(x,\partial O_0)>\delta\}$. Since $t_0\in (0,\frac{\pi}{2})$ and $\cot t_0=4b\pi$, we have that for $b>0$ large enough,
\[
t_0=\arctan \frac{1}{4b\pi}=\frac{1}{4b\pi}+O(\frac{1}{b^3}), \quad \sin t_0=\frac{1}{4b\pi}+O(\frac{1}{b^3}).
\]
Then we have
  \begin{equation}\label{eq2-35y}
\mathcal{H}(x^{L/2})<-\frac{1}{2\pi}\ln (2\sin t_0)+2bt_0=\frac{1}{2\pi}\ln b+\frac{1}{2\pi}\ln 2\pi+\frac{1}{2\pi}+O(\frac{1}{b^2}).
\end{equation}

Hence, we finally obtain from \eqref{eq2-34y}-\eqref{eq2-35y} that for fixed $0<\sigma<c_1\delta/2$, 
\[
\mathcal{H}(x^{L/2})<\inf_{x\in \Omega_\delta \setminus \Omega_L} \mathcal{H}(x)
\]
if $b>0$ is large enough, where $\Omega_L$ is given in the proof of $(i)$. Then, $(iii)$ follows. 
\end{proof}

\section{Existence and profile of maximizers} 

In this section we will establish the existence of solution for problem \eqref{eq1-8}-\eqref{eq1-9} via a variational argument.  Suppose that  $b,\Gamma$ satisfy the condition in Theorem~\ref{thm1-4}. 
Let $L$ be a fixed large constant so that $\mathcal{Q}\subset \Omega_{L/2}$, where $\mathcal{Q}$ is given in $(i)$ of Theorem~\ref{th3-1}, 
\[
\Omega_L=\{
x\in \Omega|\; -L<x_1<L
\}.
\]
Let $0<\varepsilon <\sqrt{|\Omega_L|}$ be a parameter, where $|\Omega_L|$ denotes the Lebesgue measure of $\Omega_L$.  Set 
\[
\mathcal{A}_\varepsilon:=\Big\{
\omega\in L^\infty(\Omega)\big| \int_\Omega \omega\; dx=1,\;\; 0\leq \omega\leq \frac{1}{\varepsilon^2},\;\; \text{supp} (\omega)\subset \Omega_L
\Big\}.
\]
Define the functional $\mathcal{E} : \mathcal{A}_\varepsilon\to \mathbb{R}$,
\begin{equation}
\mathcal{E} (\omega)=\frac{1}{2}\int_{\Omega_L} \omega \mathcal{G} \omega \;dx-\int_{\Omega_L} \omega \eta \;dx.
\end{equation}
For each $\omega\in \mathcal{A}_\varepsilon$,  by the $L^p$ estimate we have $\mathcal{G} \omega\in W^{2,p}(\Omega_L)$ for any $1<p<+\infty$.  As a result, $\mathcal{E}$ is well-defined on $\mathcal{A}_\varepsilon$.

\begin{lemma}\label{lem4-1} There exists $\omega_\varepsilon\in \mathcal{A}_\varepsilon$ such that 
\begin{equation}\label{eq4-2}
\mathcal{E} (\omega_\varepsilon)=\max_{\omega \in \mathcal{A}_\varepsilon}   \mathcal{E}_\varepsilon (\omega)<+\infty.
\end{equation}

Moreover, there exists a Lagrange multiplier $\mu_\varepsilon\geq -b\pi$ such that 
\begin{equation}\label{eq4-3}
\omega_\varepsilon=\frac{1}{\varepsilon^2} \mathbf{1}_{\{
\psi_\varepsilon>0
\}}\;\;\; a.e. \;\;\text{in}\;\; \Omega_L,
\end{equation}
where 
\[
\psi_\varepsilon=\mathcal{G} \omega_\varepsilon-\eta-\mu_\varepsilon. 
\]
\end{lemma}

\begin{proof}
By the expansion \eqref{eq2-4}, we have $G(\cdot,\cdot)\in L^1(\Omega_L\times \Omega_L)$. Then for any $\omega\in \mathcal{A}_\varepsilon$, 
\[
\begin{split}
\mathcal{E} (\omega)=&\frac{1}{2}\int_{\Omega_L}\int_{\Omega_L} G(y,x)\omega(y)  \omega(x) \;dxdy-\int_{\Omega_L} \omega(x) \eta(x) \;dx
\\
\leq & \frac{1}{2\varepsilon^4}\|G(\cdot,\cdot)\|_{L^1(\Omega_L\times \Omega_L)}+\lambda_\beta,
\end{split}
\]
where we have used $-\lambda_\beta\leq \eta\leq b\pi$ on $\bar{\Omega}$, and hence on $\bar{\Omega}_L$. Thus, $\max_{\omega \in \mathcal{A}_\varepsilon}   \mathcal{E} (\omega)<+\infty$. 
Let us take a maximizing sequence $\{\omega_n\} \subset \mathcal{A}_\varepsilon$ such that as $n\to +\infty$, 
\[
\mathcal{E} (\omega_n)\to  \max_{\omega \in \mathcal{A}_\varepsilon}   \mathcal{E} (\omega).
\]
Since $\mathcal{A}_\varepsilon$ is closed in $L^\infty(\Omega_L)$ weak star topology, there exists $\omega\in L^\infty(\Omega_L)$ such that 
\[
\begin{split}
&\omega_n(x)\to \omega (x)\;\;\; \text{weakly star in}\;\; L^\infty(\Omega_L),
\\
&\omega_n(x)\omega_n(y)\to \omega(x)\omega(y)\;\;\; \text{weakly star in}\;\; L^\infty(\Omega_L \times \Omega_L).
\end{split}
\]
It is easy to check that $\omega\in \mathcal{A}_\varepsilon$.   Note that  $G(\cdot,\cdot) \in L^1(\Omega_L \times \Omega_L)$ and $\eta\in L^1(\Omega_L)$. Then we conclude that  $\mathcal{E} (\omega)=\lim_{n\to +\infty} \mathcal{E} (\omega_n)=\max_{\mathcal{A}_\varepsilon}  \mathcal{E}$. So, \eqref{eq4-2} holds.

We now begin to  prove \eqref{eq4-3}. For arbitrary $\tilde{\omega}\in \mathcal{A}_\varepsilon$, define a path from $[0,1]$ to $\mathcal{A}_\varepsilon$ so that 
\[
\omega^s=\omega+s(\tilde{\omega}-\omega), \;\;\; s\in [0,1]. 
\]
Since $\omega$ is a maximizer, we have $\mathcal{E} (\omega)\geq \mathcal{E} (\omega^s)$, which gives  
\[
0\geq  \frac{d}{ds}\big|_{s=0^+} \mathcal{E} (\omega^s)=\int_{\Omega_L} (\tilde{\omega}-\omega)\big( \mathcal{G}\omega-\eta\big)\; dx.
\]
So for any $\tilde{\omega}\in \mathcal{A}_\varepsilon$, 
\[
\int_{\Omega_L} \omega\big( \mathcal{G}\omega-\eta\big)\; dx\geq \int_{\Omega_L} \tilde{\omega}\big( \mathcal{G}\omega-\eta\big)\; dx. 
\]
By the bathtub principle, we have 
\begin{equation}\label{eq4-4}
\omega=\frac{1}{\varepsilon^2}\mathbf{1}_{ \{\mathcal{G}\omega-\eta-\mu_\varepsilon
>0\}  } +\frac{c_\varepsilon}{\varepsilon^2}\mathbf{1}_{ \{\mathcal{G}\omega-\eta-\mu_\varepsilon
=0\}  }\;\;\; \text{a.e. in}\;\; \Omega_L,
\end{equation}
where 
\[
\mu_\varepsilon=\inf\Big\{
s\in \mathbb{R} \; |\; \big|\big\{    
x\in \Omega_L\;|\; \mathcal{G}\omega-\eta>s
\big\}  \big|\leq \varepsilon^2
\Big\} 
\]
and 
\[
c_\varepsilon\big| \big\{    
x\in \Omega_L\;|\; \mathcal{G}\omega-\eta=\mu_\varepsilon
\big\}\big|=\varepsilon^2-\big| \big\{    
x\in \Omega_L\;|\; \mathcal{G}\omega-\eta>\mu_\varepsilon
\big\}\big|. 
\]

On the set $\big\{    
x\in \Omega_L\;|\; \mathcal{G}\omega-\eta=\mu_\varepsilon
\big\}$, we have 
\[
\omega=-\Delta \mathcal{G}\omega=-\Delta (\eta+\mu_\varepsilon)=0,
\]
which, together with \eqref{eq4-4}, gives \eqref{eq4-3}.

Finally, we prove that $\mu_\varepsilon\geq -b\pi$. We suppose by contradiction that $\mu_\varepsilon<-b\pi$.  Since by the expansion \eqref{eq2-4} $\mathcal{G}$ is a positive elliptic operator, we have $\mathcal{G}\omega>0$ in $\Omega_L$.  So for any $x\in \Omega_L$,  
\[
\mathcal{G} \omega_\varepsilon(x)-\eta(x)-\mu_\varepsilon>-\eta(x)-\mu_\varepsilon>-b\pi-\mu_\varepsilon>0.
\]
By \eqref{eq4-3}, we have $\omega=\frac{1}{\varepsilon^2}$ a.e. in $\Omega_L$. Thus, $\int_{\Omega_L} \omega\; dx= \frac{|\Omega_L|}{\varepsilon^2}$, which is a contradiction to $\int_{\Omega_L} \omega\; dx=1$.
\end{proof}

Next, we give a lower bound of $\mathcal{E} (\omega_\varepsilon)$. 

\begin{lemma}\label{lem4-2} There exists  $C>0$ such that for all $\varepsilon>0$ small enough, 
\[
\mathcal{E} (\omega_\varepsilon)\geq \frac{1}{4\pi}\ln \frac{1}{\varepsilon}-C.
\]
\end{lemma}

\begin{proof}
Take $x_0\in \Omega_L$ and define 
\[
\tilde{\omega}_\varepsilon=\frac{1}{\varepsilon^2} \mathbf{1}_{B_{\varepsilon/\sqrt{\pi}}(x_0)}.
\]
Then for $\varepsilon>0$ small enough, we have $\tilde{\omega}_\varepsilon\in \mathcal{A}_\varepsilon$.  By \eqref{eq2-4}, we have that  $H(x,y)=H_0(x,y)-\lambda_0 \rho(x)\rho(y)$ is bounded in $B_{\varepsilon/\sqrt{\pi}}(x_0)\times B_{\varepsilon/\sqrt{\pi}}(x_0)$. So, we have 
\[
\begin{split}
\mathcal{E}(\tilde{\omega}_\varepsilon)=&\frac{1}{2}\int_{\Omega_L}\int_{\Omega_L} G(y,x) \tilde{\omega}_\varepsilon(x)\tilde{\omega}_\varepsilon(y)dx dy-\int_{\Omega_L}\eta(x)\tilde{\omega}_\varepsilon(x)dx
\\
=&\frac{1}{2}\int_{B_{\varepsilon/\sqrt{\pi}}(x_0)}\int_{B_{\varepsilon/\sqrt{\pi}}(x_0)} G(y,x) \tilde{\omega}_\varepsilon(x)\tilde{\omega}_\varepsilon(y)dx dy-\int_{B_{\varepsilon/\sqrt{\pi}}(x_0)}\eta(x)\tilde{\omega}_\varepsilon(x)dx
\\
\geq& \frac{1}{4\pi \varepsilon^4}\int_{B_{\varepsilon/\sqrt{\pi}}(x_0)}\int_{B_{\varepsilon/\sqrt{\pi}}(x_0)}\ln \frac{1}{|y-x|} dx dy-C
\\
\geq&  \frac{1}{4\pi \varepsilon^4} \big(\pi (\varepsilon/\sqrt{\pi})^2\big)^2\ln \frac{1}{\varepsilon}-C=\frac{1}{4\pi}\ln \frac{1}{\varepsilon}-C.
\end{split}
\]
We complete the proof. 
\end{proof}

We also have the following lower bounded estimate of the Lagrange multiplier $\mu_\varepsilon$. 

\begin{lemma}\label{lem4-3} There exists  $C>0$ such that for all $\varepsilon>0$ small enough, 
\[
\mu_\varepsilon\geq 2\mathcal{E} (\omega_\varepsilon)-C.
\]
As a consequence, 
\[
\mu_\varepsilon\geq \frac{1}{2\pi}\ln \frac{1}{\varepsilon}-C. 
\]
\end{lemma}

\begin{proof} 
Let  us introduce an auxiliary  function 
\[
\tilde{\psi}_\varepsilon(x)=\int_{\Omega} G_0(y,x) \omega_\varepsilon(y)\;dy-\mu_\varepsilon-b\pi.
\]
By \eqref{eq2-4}, we have  
\[
\psi_\varepsilon(x)=\tilde{\psi}_\varepsilon(x)+\lambda_0 \Big(\int_{\Omega_L} \rho(y)\omega_\varepsilon(y)\;dy\Big)\rho(x)+(b\pi-\eta(x)).
\]
Since $0<\rho<1$, $-\lambda_\beta<\beta<b\pi$ in $\Omega$, we have $\psi_\varepsilon(x)>\tilde{\psi}_\varepsilon(x)$ for any $x\in \Omega_L$. Then,
\[
\tilde{D}_\varepsilon:=\{x\in \Omega_L\;|\; \tilde{\psi}_\varepsilon(x)>0\}\subset \{x\in \Omega_L\;|\; \psi_\varepsilon(x)>0\}=\text{supp} (\omega_\varepsilon),
\]
where the last equality relation is obtained by \eqref{eq4-3}.  So we get $|\tilde{D}_\varepsilon|\leq |\text{supp} (\omega_\varepsilon)|=\varepsilon^2$.

Since $\mu_\varepsilon \geq -b\pi$, we have   $\tilde{\psi}_\varepsilon(x)=-(\mu_\varepsilon+b\pi)<0$ for $x\in \partial \Omega$,  and as  $|x|\to +\infty$, $\tilde{\psi}_\varepsilon(x)\to-(\mu_\varepsilon+b\pi)<0$.  Thus, $(\tilde{\psi}_\varepsilon)_+\in H_0^1(\Omega)$. Testing the equation $-\Delta \tilde{\psi}_\varepsilon=\omega_\varepsilon$ in $\Omega$ by   $(\tilde{\psi}_\varepsilon)_+$, we obtain that 
\[
\int_{\Omega} |\nabla (\tilde{\psi}_\varepsilon)_+|^2 dx =\int_{\Omega_L} \omega_\varepsilon(\tilde{\psi}_\varepsilon)_+dx.
\]
By the Holder inequality and the Sobolev imbedding $W^{1,1}(\Omega_L) \subset L^2(\Omega_L)$, 
\[
\begin{split}
\int_{\Omega_L} \omega_\varepsilon(\tilde{\psi}_\varepsilon)_+dx=& \frac{1}{\varepsilon^2} \int_{\tilde{D}_\varepsilon} (\tilde{\psi}_\varepsilon)_+dx\leq  \frac{|\tilde{D}_\varepsilon|^{\frac{1}{2}}}{\varepsilon^2} \left( \int_{\tilde{D}_\varepsilon} (\tilde{\psi}_\varepsilon)_+^2dx\right)^{\frac{1}{2}} 
\\
\leq & \frac{C|\tilde{D}_\varepsilon|^{\frac{1}{2}}}{\varepsilon^2}  \int_{\Omega_L} |\nabla (\tilde{\psi}_\varepsilon)_+|dx= \frac{C|\tilde{D}_\varepsilon|^{\frac{1}{2}}}{\varepsilon^2}  \int_{\tilde{D}_\varepsilon} |\nabla (\tilde{\psi}_\varepsilon)_+|dx
\\
\leq& \frac{C|\tilde{D}_\varepsilon|}{\varepsilon^2} \left( \int_{\tilde{D}_\varepsilon} |\nabla(\tilde{\psi}_\varepsilon)_+|^2dx\right)^{\frac{1}{2}} 
\\
\leq& C\left( \int_{\Omega_L} |\nabla (\tilde{\psi}_\varepsilon)_+|^2 dx \right)^{\frac{1}{2}}.
\end{split}
\]
So we get 
\[
\int_{\Omega_L} |\nabla (\tilde{\psi}_\varepsilon)_+|^2 dx\leq \int_{\Omega} |\nabla (\tilde{\psi}_\varepsilon)_+|^2 dx =\int_{\Omega_L} \omega_\varepsilon(\tilde{\psi}_\varepsilon)_+dx \leq C\left( \int_{\Omega_L} |\nabla (\tilde{\psi}_\varepsilon)_+|^2 dx \right)^{\frac{1}{2}},
\]
which gives $\int_{\Omega_L} |\nabla (\tilde{\psi}_\varepsilon)_+|^2 dx \leq C$.  As a result,  we prove that 
\[
\int_{\Omega_L} \omega_\varepsilon(\tilde{\psi}_\varepsilon)_+dx\leq C.
\]

Note that  
\[
\mathcal{G} \omega_\varepsilon (x)=\tilde{\psi}_\varepsilon(x)+\lambda_0 \Big(\int_{\Omega_L} \rho(y)\omega_\varepsilon(y)\;dy\Big)\rho(x)+\mu_\varepsilon+b\pi.
\]
So we have 
\[
\begin{split}
2\mathcal{E} (\omega_\varepsilon)=&\int_{\Omega_L} \omega_\varepsilon \mathcal{G} \omega_\varepsilon dx-2\int_{\Omega_L} \eta \omega_\varepsilon dx
\\
=&\int_{\Omega_L} \omega_\varepsilon \tilde{\psi}_\varepsilon dx+\lambda_0  \Big(\int_{\Omega_L} \rho \omega_\varepsilon\;dx\Big)^2-2\int_{\Omega_L} \eta \omega_\varepsilon dx+(\mu_\varepsilon+b\pi)\int_{\Omega_L} \omega_\varepsilon dx
\\
\leq& \int_{\Omega_L} \omega_\varepsilon (\tilde{\psi}_\varepsilon)_+ dx +\mu_\varepsilon +C
\\
\leq& \mu_\varepsilon +C.
\end{split}
\]
We complete the proof. 
\end{proof}

We begin to estimate the diameter of the support set of $\omega_\varepsilon$. To this purpose,  we first introduce a useful lemma. 

\begin{lemma}[Lemma 2.8]\label{lem4-4} Let $D\subset \Omega_L$, $0<\varepsilon<1$, $\sigma\geq 0$, and  let non-negative $\omega\in L^1(\Omega_L)$, $\int_{\Omega_L} \omega\; dx=1$ and $\|\omega\|_{L^p(\Omega_L)}\leq C_1\varepsilon^{-2(1-\frac{1}{p})}$ for some $1<p\leq +\infty$ and $C_1>0$.  Suppose for any $x\in D$, there holds  
\[
(1-\sigma)\ln \frac{1}{\varepsilon}\leq \int_{\Omega_L} \ln \frac{1}{|y-x|}\omega(y) dy+C_2,
\]
where $C_2$ is a positive constant. Then there exists some constant $R>1$ such that 
\[
\text{diam} (D)\leq R\varepsilon^{1-2\sigma}. 
\]
The constant $R$ may depend on $C_1$, $C_2$, but not on $\sigma, \varepsilon$.
\end{lemma}

We have the following estimate of  diameter of the support set of $\omega_\varepsilon$.

\begin{lemma}\label{lem4-5} There exists  $R>1$, independent of $\varepsilon$, such that 
\[
\text{diam} (\text{supp}(\omega_\varepsilon))\leq R \varepsilon. 
\]
\end{lemma}

\begin{proof}
For each $x\in supp(\omega_\varepsilon)$, we have $\psi_\varepsilon(x)>0$, namely, 
\[
\mathcal{G} \omega_\varepsilon (x)-\eta(x)-\mu_\varepsilon>0.
\]
Since $-\lambda_\beta<\eta<b\pi$, by Lemma~\ref{lem4-3} we have 
\[
\mathcal{G} \omega_\varepsilon (x)>\eta(x)+\mu_\varepsilon\geq \frac{1}{2\pi}\ln \frac{1}{\varepsilon}-C. 
\]

On the other hand, for $x\in supp(\omega_\varepsilon)$ we have 
\[
\mathcal{G} \omega_\varepsilon (x)\leq \frac{1}{2\pi} \int_{\Omega_L} \ln \frac{1}{|y-x|} \omega_\varepsilon(y) \; dy +C,
\]
where we have used $G(y,x)\leq \frac{1}{2\pi} \ln \frac{1}{|y-x|}+C$ in $\Omega_L\times \Omega_L$ in view of  \eqref{eq2-4}.  Then the assertion follows from Lemma~\ref{lem4-4}.
\end{proof}

We now turn to establish the limiting  location of $supp(\omega_\varepsilon)$.  Define the center  by 
\[
x_\varepsilon =\int_{\Omega_L} x\omega_\varepsilon(x)\;  dx.
\]
We may assume that up to a subsequence, 
\[
x_\varepsilon\to x^*\in \bar{\Omega}_L\;\;\;\text{as} \;\;\varepsilon\to 0_+. 
\]

\begin{lemma}\label{lem4-6} It holds
\[
\mathcal{H} (x^*)=\min_{x\in \Omega} \mathcal{H}(x). 
\]
As a consequence, $dist (x^*, \partial \Omega_L)>0$. 
\end{lemma}

\begin{proof} Fix $x_0\in \Omega_L$. Then,   $dist (x_0,\partial \Omega_L)>0$.  We set $\tilde{\omega}_\varepsilon(\cdot) =\omega_\varepsilon (x_\varepsilon-x_0+\cdot)$. By Lemma~\ref{lem4-5}, we have 
$supp (\omega_\varepsilon)\subset B_{R\varepsilon}(x_\varepsilon)$. Thus, $supp (\tilde{\omega}_\varepsilon)\subset B_{R\varepsilon}(x_0)$. So for small $\varepsilon>0$,  $supp (\tilde{\omega}_\varepsilon) \subset B_{R\varepsilon}(x_0) \subset \Omega_L$. Then it is easy to see that $\tilde{\omega}_\varepsilon \in \mathcal{A}_\varepsilon$.  Since $\omega_\varepsilon$ is a maximizer, we have $\mathcal{E}(\omega_\varepsilon)\geq \mathcal{E}(\tilde{\omega}_\varepsilon)$. It is clear that 
\[
\frac{1}{2\pi} \int_{\Omega_L}\int_{\Omega_L}  \ln \frac{1}{|y-x|} \omega_\varepsilon(x)\omega_\varepsilon(y) \; dxdy=\frac{1}{2\pi} \int_{\Omega_L}\int_{\Omega_L}  \ln \frac{1}{|y-x|} \tilde{\omega}_\varepsilon(x)\tilde{\omega}_\varepsilon(y) \; dxdy. 
\]
So we obtain that 
\[
\begin{split}
\frac{1}{2}\int_{\Omega_L}\int_{\Omega_L} H(y,x) &\omega_\varepsilon(x)\omega_\varepsilon(y) \; dxdy+\int_{\Omega_L} \eta(x) \omega_\varepsilon(x)\; dx
\\
\leq &\frac{1}{2}\int_{\Omega_L}\int_{\Omega_L} H(y,x) \tilde{\omega}_\varepsilon(x)\tilde{\omega}_\varepsilon(y) \; dxdy+\int_{\Omega_L} \eta(x) \tilde{\omega}_\varepsilon(x)\; dx.
\end{split}
\]
By passing $\varepsilon\to 0_+$, we get $\mathcal{H}(x^*)\leq \mathcal{H}(x_0)$. Thus, $\mathcal{H} (x^*)=\min_{x\in \Omega_L} \mathcal{H}(x)$.  By Theorem~\ref{th3-1},  we have $\mathcal{H} (x^*)=\min_{x\in \Omega} \mathcal{H}(x)$ and $|x^*_1|\leq L/2$. So,  $dist (x^*, \partial \Omega_L)>0$. 
\end{proof}

By Lemma~\ref{lem4-5} and Lemma~\ref{lem4-6}, we have 

\begin{corollary}\label{cor4-7} For $\varepsilon>0$ small enough,  there holds $dist( supp(\omega_\varepsilon), \partial \Omega_L)>0$. That is, the free boundary belonging to $\omega_\varepsilon$ does not touch $\partial \Omega_L$.  Following the same arguments as in \cite[Theorem~5.2]{Tur83}, we have that for all $\phi\in C_0^\infty(\Omega_L)$, 
\[
\int_{\Omega_L} \omega_\varepsilon \nabla \phi \cdot \nabla^\perp (\mathcal{G}\omega_\varepsilon-\eta)\;dx=0.
\]
Since $supp(\omega_\varepsilon) \subset \Omega_L$ and $dist( supp(\omega_\varepsilon), \partial \Omega_L)>0$, it is easy to see that $\omega_\varepsilon$ is a weak solution of  \eqref{eq1-8}-\eqref{eq1-9}. 
\end{corollary}

Moreover, we have
\begin{proposition}\label{pro4-8} For all sufficiently small $\varepsilon>0$, it holds 
\[
\omega_\varepsilon=\frac{1}{\varepsilon^2} \mathbf{1}_{\{
\psi_\varepsilon>0
\}}\;\;\; a.e. \;\;\text{in}\;\; \Omega.
\]
\end{proposition}

\begin{proof} By \eqref{eq4-3}, we only need  to show that  $\psi_\varepsilon\leq 0$ a.e. in $\Omega \setminus \Omega_L$.  By \eqref{eq4-3} again, we have $\psi_\varepsilon (x)\leq 0$ for $x\in \Omega$ and $x_1=\pm L$.  By Lemma~\ref{lem4-3}, for $\varepsilon>0$ small enough we have $\psi_\varepsilon=-\eta-\mu_\varepsilon<0$ on $\partial S$ and  $\psi_\varepsilon\to-\eta-\mu_\varepsilon<0$ as $|x|\to \infty$.  Since $supp(\omega_\varepsilon) \subset \Omega_L$, by the definition of $\psi_\varepsilon$ we have 
\[
\Delta \psi_\varepsilon =0\quad \text{in} \;\;\Omega \setminus \Omega_L.
\]
The maximum principle implies that $ \psi_\varepsilon \leq 0$ in $\Omega \setminus \Omega_L$.
\end{proof}

\begin{lemma}\label{lem4-9} As $\varepsilon\to 0_+$, there holds 
\begin{equation}\label{eq4-5}
\mathcal{E} (\omega_\varepsilon)= \frac{1}{4\pi}\ln \frac{1}{\varepsilon}+O(1),
\end{equation}
\begin{equation}\label{eq4-6}
\mu_\varepsilon= \frac{1}{2\pi}\ln \frac{1}{\varepsilon}+O(1).
\end{equation}
\end{lemma}

\begin{proof}By the Riesz's rearrangement inequality, we have 
\[
\begin{split}
 \int_{\Omega_L}\int_{\Omega_L}  \ln \frac{1}{|y-x|} \omega_\varepsilon(x)\omega_\varepsilon(y) \; dxdy\leq&  \int_{\Omega_L}\int_{\Omega_L}  \ln \frac{1}{|y-x|} \tilde{\omega}_\varepsilon(x)\tilde{\omega}_\varepsilon(y) \; dxdy
 \\
 \leq &  \frac{1}{4\pi}\ln \frac{1}{\varepsilon}+C,
 \end{split}
\]
where $ \tilde{\omega}_\varepsilon=\frac{1}{\varepsilon^2}\mathbf{1}_{B_{\varepsilon/\sqrt{\pi}}(x_\varepsilon)}$. This shows  
\[
\mathcal{E} (\omega_\varepsilon)\leq \frac{1}{4\pi}\ln \frac{1}{\varepsilon}+C, 
\]
which, together with Lemma~\ref{lem4-2}, gives \eqref{eq4-5}. 

Next we prove \eqref{eq4-6}. 

Let $\tilde{\psi}_\varepsilon\in H_0^1(\Omega)$ be defined by 
\[
\tilde{\psi}_\varepsilon(x)=\int_{\Omega} G_0(y,x) \omega_\varepsilon(y)\;dy.  
\]
Let $\delta>0$ be a small constant.  Since $supp (\Omega_\varepsilon)\subset B_{R\varepsilon}(x_\varepsilon)$, 
using a similar argument as in \cite[Lemma 2.1]{CGPY}, we can prove that for $x\in \Omega\setminus B_\delta (x_\varepsilon)$,  
\[
\frac{\partial \tilde{\psi}_\varepsilon(x)}{\partial x_i}=\frac{\partial G_0(x_\varepsilon,x)}{\partial x_i}+O(\varepsilon),\;\;\; i=1,2. 
\]
Note that 
\[
\psi_\varepsilon(x)=\tilde{\psi}_\varepsilon(x)+\lambda_0 \Big(\int_{\Omega_L} \rho(y)\omega_\varepsilon(y)\;dy\Big)\rho(x)-\eta(x)-\mu_\varepsilon.
\]
So fior $x\in \Omega\setminus B_\delta (x_\varepsilon)$,
\begin{equation}\label{eq4-7}
\frac{\partial \psi_\varepsilon(x)}{\partial x_i}=\frac{\partial G_0(x_\varepsilon,x)}{\partial x_i}+O(1).
\end{equation}
Since $ \psi_\varepsilon$ satisfies the following equation 
\[
-\Delta \psi_\varepsilon=\frac{1}{\varepsilon^2}  \mathbf{1}_{\{
\psi_\varepsilon>0
\}}\;\;\; \text{in}\;\;\Omega, 
\]
we have the following Pohozaev identity 
\[
\begin{split}
\int_{\partial B_\delta(x_\varepsilon)} &\big( (x-x_\varepsilon)\cdot \nabla \psi_\varepsilon \big)\frac{\partial \psi_\varepsilon}{\partial \mathbf{n}}\;dS_x-\frac{1}{2}\int_{\partial B_\delta(x_\varepsilon)} \big((x-x_\varepsilon) \cdot \mathbf{n} \big)|\nabla \psi_\varepsilon|^2\;dS_x
\\
=&2\frac{1}{\varepsilon^2} \int_{\{
\psi_\varepsilon>0
\}} \psi_\varepsilon\;  dx=2\int_{\Omega_L} \omega_\varepsilon \psi_\varepsilon \;dx. 
\end{split}
\]
Then by \eqref{eq4-7}, we have 
\begin{equation}\label{eq4-8}
\int_{\Omega_L} \omega_\varepsilon \psi_\varepsilon \;dx=O(1). 
\end{equation}

So we have 
\[
\begin{split}
2\mathcal{E} (\omega_\varepsilon)=\int_{\Omega_L} \omega_\varepsilon \mathcal{G} \omega_\varepsilon dx-2\int_{\Omega_L} \eta \omega_\varepsilon dx
=\int_{\Omega_L} \omega_\varepsilon \psi_\varepsilon dx-\int_{\Omega_L} \eta \omega_\varepsilon dx+\mu_\varepsilon
=\mu_\varepsilon+O(1).
\end{split}
\]
Thus, \eqref{eq4-6} follows from \eqref{eq4-5}.
\end{proof}

\begin{proof}[{\bf Proof of Theorem~\ref{thm1-4}}] The conclusions follows from Lemma~\ref{lem4-5}, Lemma~\ref{lem4-6}, Corollary~\ref{cor4-7}, Proposition~\ref{pro4-8} and Lemma~\ref{lem4-9}. 
\end{proof}

\begin{proof}[{\bf Proof of Theorem~\ref{thm1-5}}]  We only need to replace $\Omega_L$ by $D_{b,\Gamma}$, in the proof of Theorem~\ref{thm1-4},   to obtain the existence of vortex-patch solutions  of  \eqref{eq1-8}-\eqref{eq1-9} , where  $D_{b,\Gamma}$ is defined in $(ii)$ of Theorem~\ref{th3-1}. 
\end{proof}

\appendix
\section{A harmonic problem}

In this appendix we consider the existence and uniqueness of solution for following harmonic equation with prescribed circulation 
\begin{equation}\label{eqA-1}
\begin{cases}
\Delta u=0\quad \text{in}\; \Omega,\\
u=f+\lambda \;\;  \text{on}\;  \partial O_0,
\;\;\;
u=0 \;\; \text{on}\;\;  \partial S, 
\\
u \to 0\;\; \text{as}\;\; |x|\to \infty, 
\\
\int_{\partial O_0} \frac{\partial u}{\partial \mathbf{n}} dS_x=\Lambda, 
\end{cases}
\end{equation}
where $f$ is a given function,  $\Lambda\in \mathbb{R}$ is a constant.  The pair $(u,\lambda)$ is unknown for function $u$  and flux $\lambda$.

Take a small $0<\delta< \frac{1}{2}\text{dist}(\partial S,\partial O_0)$, where 
$
\text{dist}(\partial S,\partial O_0):=\inf_{x\in \partial S, \; y\in \partial O_0} |x-y|. 
$
Define 
\[
(\partial O_0)_\delta=\{ x\in \Omega|\; \text{dist}(x,\partial O_0)<\delta
\}.
\]
Then we have 

\begin{lemma}\label{lemA-1} Suppose that $f\in C^{2,\alpha}(\overline{(\partial O_0)_\delta})$ for some $\alpha\in (0,1)$. Let $\Lambda\in \mathbb{R}$  be a constant. 
Then \eqref{eqA-1} has a unique solution $(u,\lambda)\in C^{2,\alpha}(\bar{\Omega})\times \mathbb{R}$, satisfying the Schauder estimate
\[
\|u\|_{C^{2,\alpha} (\bar{\Omega})}\leq C (  \|f\|_{C^{2,\alpha}(\overline{(\partial O_0)_\delta})} +|\lambda| ). 
\]
where $C>0$ is a constant, independent of $u,f,\lambda$. 

Moreover, as $|x|\to \infty$, $|\nabla u|\to 0$. 
\end{lemma}


\begin{proof}
Define the functional 
\[
J(u)=\frac{1}{2}\int_\Omega |\nabla u|^2 dx-\Lambda\left( \frac{1}{|\partial O_0|}\int_{\partial O_0} u\; dS_x\right). 
\]
We will find a critical point of $J$ in the following admissible set 
\[
K=\{ u\in H^1(\Omega) |\; u=0\;\; \text{on}\;\; \partial S,\quad   u=f+\text{constant} \;\; \text{on}\;\; \partial O_0  \}. 
\]

{\bf Step 1.} There exists a minimizer of $J$ over $K$.

By the Poincar\'e inequality, there exist a constant $C>0$, independent  of $u$, such that 
\begin{equation}\label{eqA-2}
\|u \|_{L^2(\Omega)}\leq C  \|\nabla u\|_{L^2(\Omega)}\quad \forall u \in K.
\end{equation}
Then by the H\"older inequality and the trace theorem, we have  
\begin{equation}\label{eqA-3}
\big| \int_{\partial O_0} u \;dS_x \big|\leq C \|u\|_{L^2(\partial O_0)}\leq C \|u\|_{H^1(\Omega)}\leq C \|\nabla u\|_{L^2(\Omega)}. 
\end{equation}
As a result, 
\begin{equation}\label{eqA-4}
J(u)\geq \frac{1}{2} \|\nabla u\|_{L^2(\Omega)}^2-C \|\nabla u\|_{L^2(\Omega)}, 
\end{equation}
which implies that  $J$ is bonded from below on $K$. Then, we  shall  consider the following minimization problem 
\[
c_K:=\min_{\beta \in K} J(\beta). 
\]

Let $\{u_n\}\subset H^1(\Omega)$  with $u_n=0$ on $\partial S$,   $u_n=f+c_n$ on $\partial O_0$ be a minimization sequence such that 
\[
J(u_n)\to c_K\quad \text{as}\;\;n\to \infty,
\]
where $\{ c_n\} \subset \mathbb{R}$. 
By \eqref{eqA-2} and \eqref{eqA-4}, $\{u_n\}$ is bounded in $H^1(\Omega)$.  Hence, by \eqref{eqA-3} $\{ c_n\}$ is bounded. So there exists $(u,\lambda)\in H^1(\Omega)\times \mathbb{R}$  with $u=0$ on $\partial S$,   $u=f+\lambda$ on $\partial O_0$,  such that as $n\to +\infty$,
\[
c_n\to \lambda ,\quad  u_n \rightharpoonup u  \quad \text{weakly in}\;\; H^1(\Omega).
\]
It follows from 
\[
\int_\Omega |\nabla u|^2 dx\leq \liminf_{n\to+\infty} \int_\Omega |\nabla u_n|^2 dx
\]
that 
\[
J(u)\leq \liminf_{n\to+\infty} J(u_n)=c_K. 
\]

On the other hand, since $u\in K$, we have $J(u)\geq c_K$. So $u$ is a minimizer of $J$ on $K$.

\smallskip 

{\bf Step 2.}   $u$ is a solution of \eqref{eqA-1}. 

For any $\xi\in C_0^\infty (\Omega)$ and $t\in \mathbb{R}$, we have $u+t\xi\in K$. Then 
\[
0=\frac{d}{dt} J(u+t\xi)=\int_\Omega \nabla u\cdot \nabla \xi dx.
\]
Thus $u$ is a weak solution of 
\begin{equation}\label{eqA-5}
\begin{cases}
\Delta u =0\;\;\;\text{in}\;\;\Omega,
\\
u=f+\lambda\;\;\;\text{on}\;\;\partial O_0,
\;\;\;
u=0\;\;\;\text{on}\;\; \partial S.
\end{cases}
\end{equation}

 By reflections on $\partial S$, we can extend $u$ into a harmonic function $\hat{u}$ in $\Omega_\delta=(\mathbb{R}\times (-\delta,\pi+\delta))\setminus O_0$ such that $\hat{u}=u$ in $\Omega$. Since $u \in H^1(\Omega)$, we have $\hat{u} \in H^1(\Omega_\delta)$.  So  $\|\hat{u}\|_{L^2(\Omega_\delta\setminus B_R(0))}\to 0$ as $R\to +\infty$.  By the Moser iteration, for $x\in \Omega$ with large $|x|$,
\[
|u(x)|=|\hat{u}(x)|\leq C \|\hat{u}\|_{L^2(B_\delta(x))}.
\]
Then we have
\begin{equation}\label{eqA-6}
u(x)\to 0\quad \text{as}\;\; |x|\to +\infty.
\end{equation}

Next we study the regularity of $u$.   We take two simply-connected, smooth domains $W_1,W_2\subset S$, such that $O_0\cup (\partial O_0)_{\delta/2} \subset W_1 \subset W_2\subset O_0\cup (\partial O_0)_{\delta} $ and $\text{dist} (\partial W_1,\partial W_2)>0$. Since $\hat{u}$ is harmonic in $\Omega_\delta$, $\hat{u}\in C^\infty(\Omega_\delta)$. As a result, $u\in C^\infty(\overline{S\setminus W_1})$.  Take $\xi\in C_0^\infty(W_2)$ such that  $\xi= 1$ in $W_1$.  Then $v:=\xi (u-f-\lambda)\in H_0^1(W_2\setminus O_0)$ is a solution of 
\[
\Delta v=(u-f-\lambda) \Delta \xi-\xi \Delta f+2\nabla \xi \cdot \nabla (u-f-\lambda)=:g. 
\]
Since $\nabla \xi=0$ and $\Delta \xi =0$ in $W_1$, we have 
\[
\|g\|_{C^\alpha (\overline{W_2\setminus O_0})} \leq C \big(\|u\|_{C^\alpha (\overline{W_2\setminus W_1})}+   \|f\|_{C^{2,\alpha} (\overline{W_2\setminus O_0})}     \big) \leq C  \big(\|u\|_{C^\alpha (\overline{W_2\setminus W_1})}+   \|f\|_{C^{2,\alpha} (\overline{(\partial O_0)_{\delta} })}     \big).
\]
By the Schauder estimate, $v\in C^{2,\alpha}(\overline{W_2\setminus O_0})$ and 
\[
\|v\|_{C^{2,\alpha}(\overline{W_2\setminus O_0}) }\leq C  \big(\|u\|_{C^\alpha (\overline{W_2\setminus W_1})}+   \|f\|_{C^{2,\alpha} (\overline{(\partial O_0)_{\delta} })}     \big).
\]
Since $v=u-f-\lambda$ in $W_1\setminus O_0$, we get 
\begin{equation}\label{eqA-7}
\|u\|_{C^{2,\alpha}(\overline{W_1\setminus O_0}) }\leq C  \big(\|u\|_{C^\alpha (\overline{W_2\setminus W_1})}+   \|f\|_{C^{2,\alpha} (\overline{(\partial O_0)_{\delta} })} +|\lambda|    \big).
\end{equation}
Hence, we finally obtain that $u\in C^{2,\alpha}(\bar{\Omega})$.

Now we take $\xi\in C^2(\bar{\Omega})\cap H^1(\Omega)$ satisfying that $\text{supp}(\xi)$ is bounded,  $\xi=1$ on $\partial O_0$ and $\xi=0$ on $\partial S$. Then $u+t\xi \in K$ for each $t\in \mathbb{R}$. So, 
\[
\begin{split}
0=&\frac{d}{dt} J(u+t\xi)=\int_\Omega \nabla u\cdot \nabla \xi dx-\Lambda \left( \frac{1}{|\partial O_0|}\int_{\partial O_0} \xi\; dS_x\right)
\\
=&-\int_\Omega (\Delta u) \xi\;dx +\int_{\partial O_0}\frac{\partial u}{\partial \mathbf{n}} \xi\; dS_x-\Lambda\left( \frac{1}{|\partial O_0|}\int_{\partial O_0} \xi\; dS_x\right)
\\
=&\int_{\partial O_0}\frac{\partial u}{\partial \mathbf{n}} \; dS_x-\Lambda.
\end{split}
\]
Thus,  we obtain 
\begin{equation}\label{eqA-8}
 \int_{\partial O_0}\frac{\partial u}{\partial \mathbf{n}} \; dS_x=\Lambda. 
\end{equation}

Combining  \eqref{eqA-5}, \eqref{eqA-6}, \eqref{eqA-7},  $u\in C^{2,\alpha}(\bar{\Omega})$ is a classical solution of \eqref{eqA-1}.

\smallskip 

{\bf Step 3.}  The solution $u$ is unique. 

We assume that there exist two solutions $(u_1,\lambda_1)$, $(u_2,\lambda_2)$ of \eqref{eqA-1} in $(C^2(\Omega)\cap C^1(\bar{\Omega}))\times \mathbb{R}$. Let $w=u_1-u_2$. Then 
\[
\begin{cases}
\Delta w=0\quad \text{in}\; \Omega,\\
w=\lambda_1-\lambda_2 \;\;  \text{on}\; \partial O_0,\quad w=0 \;\; \text{on}\;\; \partial S, 
\\
 w\to 0\quad \text{as}\;\; |x|\to +\infty, 
\\
\int_{\partial O_0} \frac{\partial w}{\partial \mathbf{n}} =0.
\end{cases}
\]

If $\lambda_1>\lambda_2$, then by the maximum principle $0<w<\lambda_1-\lambda_2$ in $\Omega$. From the Hopf lemma, $\frac{\partial w}{\partial \mathbf{n}} >0$ on $\partial O_0$. As a result, $\int_{\partial O_0} \frac{\partial w}{\partial \mathbf{n}} dS_x>0$. This is impossible. 

If $\lambda_1<\lambda_2$, we can use a similar argument to deduce that $\int_{\partial O_0} \frac{\partial w}{\partial \mathbf{n}} dS_x<0$, which is also impossible. 

Thus, $\lambda_1=\lambda_2$. Then by the maximum principle we have $w=0$ in $\Omega$.

\smallskip 

{\bf Step 4.}  The global Schauder estimate for $u$. 

We still use the notations $\hat{u}, W_1, W_2, \Omega_\delta$ in Step 2.
By the maximum principle,
\[
\|u\|_{L^\infty(\Omega)}\leq \|f\|_{L^\infty(\partial O_0)}+|\lambda|. 
\]
So we have the following    gradient estimate for harmonic function:  for $x\in \Omega\setminus (\partial O_0)_{\delta/2}$,
\[
|Du(x)|=|D\hat{u}(x)|\leq  \frac{4}{\delta}\max_{y\in \bar{B}_{\delta/2}(x)}|\hat{u}(y)|\leq \frac{4}{\delta} \|\hat{u}\|_{L^\infty(\Omega_\delta)}\leq \frac{4}{\delta} \big(\|f\|_{L^\infty(\partial O_0)}+|\lambda|\big),
\]
where we have used $\|\hat{u}\|_{L^\infty(\Omega_\delta)}=\|u\|_{L^\infty(\Omega)}$.  Similarly, there exists a constant $C>0$, dependent only on $\delta$, such that for $x\in \Omega\setminus (\partial O_0)_{\delta/2}$, 
\[
|D^2u(x)|,\; |D^3u(x)|\leq C\big(\|f\|_{L^\infty(\partial O_0)}+|\lambda|\big). 
\]
So we obtain that 
\begin{equation}\label{eqA-9}
\|u\|_{C^3(\overline{\Omega\setminus (\partial O_0)_{\delta/2}})}\leq C\big(\|f\|_{L^\infty(\partial O_0)}+|\lambda|\big). 
\end{equation}
Then by \eqref{eqA-7}, we get 
\begin{equation}\label{eqA-10}
\begin{split}
\|u\|_{C^{2,\alpha}(\overline{W_1\setminus O_0}) }\leq& C  \big(\|u\|_{C^\alpha (\overline{W_2\setminus W_1})}+   \|f\|_{C^{2,\alpha} (\overline{(\partial O_0)_{\delta} })} +|\lambda|    \big)
\\
\leq&C  \big(\|u\|_{C^\alpha(\overline{\Omega\setminus (\partial O_0)_{\delta/2}})}+   \|f\|_{C^{2,\alpha} (\overline{(\partial O_0)_{\delta} })} +|\lambda|    \big)
\\
\leq&C  \big( \|f\|_{C^{2,\alpha} (\overline{(\partial O_0)_{\delta} })} +|\lambda|    \big).
\end{split}
\end{equation}
Combining \eqref{eqA-9}-\eqref{eqA-10}, we finally obtain that 
\begin{equation}\label{eqA-11}
\|u\|_{C^{2,\alpha}(\bar{\Omega})}\leq C  \big( \|f\|_{C^{2,\alpha} (\overline{(\partial O_0)_{\delta} })} +|\lambda|    \big).
\end{equation}

\smallskip 

{\bf Step 5.}  The asymptotic behavior for $\nabla u$.

Since $\hat{u}$ is a harmonic function in $\Omega_\delta$, we have that for $x\in \Omega$ with large $|x|$,
\[
|\nabla u(x)|=|\nabla \hat{u}(x)|\leq \frac{2}{\delta} \max_{y\in B_\delta(x)\cap \Omega}|\hat{u}(y)|. 
\]
So by \eqref{eqA-6}, $\nabla u \to (0,0)$ as $|x|\to +\infty$.

We complete the proof. 
\end{proof}


\section{The Green representation formulas}

In this appendix we prove the  following Green representation formulas. 

\begin{lemma}\label{lemC-1}Suppose that $f\in C^1(\bar{\Omega})\cap C^2(\Omega)$ is a harmonic function in $\Omega$ with $f=0$ on $\partial S$. Suppose further that $f, |\nabla f|$ are bounded. Then we have the following Green representation formula 
\begin{equation}\label{eqC-1}
f(x)=\int_{\partial O_0}\,G_S(y,x)\frac{\partial f(y)}{\partial \mathbf{n}}-f(y)\frac{\partial G_S(y,x)}{\partial \mathbf{n}}\, dS_y,\quad x\in \Omega, 
\end{equation}
\begin{equation}\label{eqC-2}
f(x)=-\int_{\partial O_0}\,f(y)\frac{\partial G_0(y,x)}{\partial \mathbf{n}}\, dS_y,\quad x\in \Omega, 
\end{equation}
where $\mathbf{n}$ is the unit outward normal of $\partial \Omega$.  


\end{lemma}

\begin{proof} For $x\in \Omega$, take a large $L>0$ such that $x\in \Omega_L$, where $\Omega_L=\{x\in \Omega|\; -L<x_1<L\}$.  By the Green representation formula in bounded domain, we have 
\begin{equation}\label{eqC-4}
\begin{split}
f(x)=&\int_{\partial \Omega_L}\,G_S(y,x)\frac{\partial f(y)}{\partial \mathbf{n}}-f(y)\frac{\partial G_S(y,x)}{\partial \mathbf{n}}\, dS_y
\\
=&\int_{\partial O_0}\,G_S(y,x)\frac{\partial f(y)}{\partial \mathbf{n}}-f(y)\frac{\partial G_S(y,x)}{\partial \mathbf{n}}\, dS_y
\\
&+\int_{\{y\in \Omega|\; |y_1|=L\}}\,G_S(y,x)\frac{\partial f(y)}{\partial \mathbf{n}}-f(y)\frac{\partial G_S(y,x)}{\partial \mathbf{n}}\, dS_y
\\
\\
=&\int_{\partial O_0}\,G_S(y,x)\frac{\partial f(y)}{\partial \mathbf{n}}-f(y)\frac{\partial G_S(y,x)}{\partial \mathbf{n}}\, dS_y+o_L(1),
\end{split}
\end{equation}
where we have used $G_S(y,x),|\nabla_y G_S(y,x)|\to 0$ as $|y|\to +\infty$.  By letting $L\to +\infty$ in \eqref{eqC-4}, we get \eqref{eqC-1}. 


The formula \eqref{eqC-2} can be obtained by a similar argument as in the proof of \eqref{eqC-1}. 
\end{proof}

{\bf Acknowledgements.} W. Yu was supported by China Postdoctoral Science Foundation (Grant  2023M730333).

\bigskip


\phantom{s}
\thispagestyle{empty}


\begin{thebibliography}{99}



\bibitem{CGPY}
D. Cao, Y. Guo, S. Peng and S. Yan, Local uinqueness for vortex patch problem in incompressible planar steady flow,
\textit{J. Math. Pures Appl.}, 131(2019), 251--289.


\bibitem{CPY2}
D. Cao, S. Peng and S. Yan, Planar vortex patch problem in incompressible steady flow, \textit{Adv. Math.}, 270(2015), 263--301.


\bibitem{CCG}
A. Castro, D. C\'orodoba and J. G\'omez-Serrino, Uniformly rotating global patch  solutions  for active scalars, \textit{Ann. PDE}, 2(1)()2016, 1--34.

\bibitem{CGO}
D. C\'ordoda, R. Graneo-Belinvh\'on and R. Orive, The confined Muskat problem: differences with the deep water regime, \textit{Commun. Math. Phys.}, 12(2014), 423--455. 




\bibitem{HHHM}
F. de la Hoz, Z. Hassainia, T. Himidi and J. Mateu, An analytical and numerical study of steady patches in the disc, \textit{Anal. PDE}, 9(7), 2016, 1609--1670.

\bibitem{HHMV}
F. de la Hoz,  T. Himidi,  J. Mateu and J. Verdera, Doubly connected V-states for the planar Euler equations, \textit{SIAM J. Math. Anal.,} 48(3)(2016), 1892--1928.




\bibitem{EFHM}
A. Elcrat, B. Fornberg, M. Horn and K. Miller, Some steady vortex flows past a circular cylinder,
\textit{J. Fluid Mech.}, 409(2000), 13--27.

\bibitem{EM}
A. R. Elcrat and K. G. Miller, Steady vortex flows with circulation past asymmetric obstacles, \textit{Comm. Partial Differential Equations}, 12:10(1987), 1095--1115.


\bibitem{F}
L. F\"oppl, ``Wirbelbewegung hinter einem Kreiszylinder," \textit{Sitzb. Bayer. Akad. Wiss.}, 1(1913), 1. 


\bibitem{HI}
J. He and D. Iftimie, On the small rigid body limit in 3D incompressible flows, \textit{J. Lond. Math. Soc.}, 104(2021), 668--687.

\bibitem{HRHB}
M. Heil, J. Rosso, A. L. Hazel and M. Brons, Topological fluid mechanics of the formation of the K\'arm\'an-vortex street, \textit{J. Fluid Mech.}, 812(2017), 199--221.

\bibitem{HM}
 T. Himidi and J. Mateu, Degenerate bifurcation of the rotating patches, \textit{Adv. Math.,} 302(2016), 799--850.

\bibitem{HYY}
Z. Huang, J. Yang and W. Yu, Boundary plasmas for a confined plasma problem in
dimensional two, \textit{Calc. Var. PDE}, 62(3), 2023.




\bibitem{ILN1}
 D. Iftimie, M. C. Lopes Filho and H. J. Nussenzveig Lopes, Two-dimensional incompressible ideal flow around a small obstacle, \textit{Comm. Partial Differential Equations}, 28(2003), 349--379.

\bibitem{ILN2}
D. Iftimie, M. C. Lopes Filho and H. J. Nussenzveig Lopes, Two-dimensional incompressible viscous flow around a small obstacle, \textit{Math. Ann.}, 336(2006), 449--489. 



\bibitem{vK}
T. von K\'arm\'an,  \"Uber den Mechanismus des Widerstandes, den ein bewegter K\"orper in einer Fl\"ussigkeit erfahrt, \textit{Nachr. Ges. Wiss. G\"ottingen}, 1912(1912), 547--556.

\bibitem{K} G. Kirchhoff,  \textit{Vorlesungen \"uber Mathematische Physik}, Leipzig: Teubner, 1876.



\bibitem{L}
C. Lacave, Two dimensional incompressible ideal flow around a thin obstacle tending to a curve, \textit{Ann. Inst. H. Poincar\'e Anal. Non Lin\'eaire}, 26(2009), 1121--1148.


\bibitem{LM}
C. Lacave and N. Masmoudi, Impermeability through a perforated domain for the incompressible two dimensional Euler equations, \textit{Arch. Ration. Mech. Anal.}, 221(2016), 1117--1160. 

\bibitem{M}
C. Marchioro, On the growth of vorticity support for an incompressible non-viscous fluid in a two-dimensional exterior domain, \textit{Math. Methods Appl. Sci.}, 19(1996), 53--62.

\bibitem{ONS}
K. Ozawa, H. Nakamura, K. Shimamura, et al., Capillary-driven horseshoe vortex forming around a micro-pillar, \textit{Journal of Colloid and Interface Science}, 642(2023), 227--234.

\bibitem{PPS}
P. Puthan, G. Pawlak and S. Sarkar, Wake vortices and dissipation in a tidally modulated flow past a three-dimensional topography, \textit{Journal of Geophysical Research: Oceans}, 127(8)(2022), e2022JC018470.

\bibitem{S}
K. B. Shah, \textit{Large eddy simulations of flow past a cubic obstacle}, Stanford University, 1998.

\bibitem{Tur83}
B. Turkington, On steady vortex flow in two dimensions. II, \textit{Comm. Partial Differential Equations}, 8(1983), 1031--1071.



\bibitem{VMS}
G. L. Vasconcelos, M. N. Moura and A. M. J. Schakel,  Vortex motion around a circular cylinder, \textit{Phys. Fluids}, 23(2011), 123601.

























	

































	





























































































	
\end{thebibliography}
\end{document}